\documentclass[reqno]{amsproc}
\usepackage{amsthm}
\usepackage{amsfonts}
\usepackage{amsmath}
\usepackage{amssymb}
\usepackage{mathrsfs}
\usepackage{euscript}   
\usepackage{color}
\usepackage{mathbbol}
\usepackage{pifont}
\usepackage{tikz}

\makeatletter
\@namedef{subjclassname@2010}{%
\textup{2010} Mathematics Subject Classification}
\makeatother

\numberwithin{equation}{section}

\newtheorem{thm}{Theorem}[section]
\newtheorem{pro}[thm]{Proposition}
\newtheorem{cor}[thm]{Corollary}
\newtheorem{lem}[thm]{Lemma}

\newtheorem*{opq*}{\bf Problem}
\newtheorem{dfn}[thm]{Definition}

\theoremstyle{remark}
\newtheorem{rem}[thm]{Remark}
\newtheorem*{bc*}{\sc Basic Construction}
\theoremstyle{definition}
\newtheorem{exa}[thm]{Example}

\newcommand*{\acal}{\EuScript{A}}
\newcommand*{\bcal}{\EuScript{B}}

\newcommand*{\cbb}{\mathbb{C}}
\newcommand*{\dbb}{\mathbb{D}}

\newcommand*{\dcal}{\EuScript{D}}
\newcommand*{\ecal}{\EuScript{E}}
\newcommand*{\ff}{\EuScript{F}}
\newcommand*{\lcal}{\EuScript{L}}
\newcommand*{\gibh}{\mathscr{I}_{\hh_1,\hh_2}}

\newcommand*{\gqbh}{\mathscr{Q}_{\hh_1,\hh_2}}
\newcommand*{\gnbh}{\mathscr{N}_{\hh_1,\hh_2}}
\newcommand*{\gsbh}{\mathscr{S}_{\hh_1,\hh_2}}
\newcommand*{\ghbh}{\mathscr{H}_{\hh_1,\hh_2}}

\newcommand*{\gsbhh}[1]{\mathscr{S}_{{#1}}}
\newcommand*{\gnbhh}[1]{\mathscr{N}_{{#1}}}
\newcommand*{\gubhh}[1]{\mathscr{U}_{{#1}}}
\newcommand*{\Ge}{\geqslant}
\newcommand*{\hh}{\EuScript{H}}

\newcommand*{\is}[2]{\langle#1,#2\rangle}
\newcommand*{\jd}[1]{\mathscr N(#1)}
\newcommand*{\kk}{\EuScript{K}}
\newcommand*{\Le}{\leqslant}
\newcommand*{\mcal}{\EuScript{M}}
\newcommand*{\ncal}{\EuScript{N}}

\newcommand*{\ob}[1]{{\mathscr R}(#1)}
\newcommand*{\ogr}[1]{\boldsymbol B(#1)}

\begin{document}
   \title[Lifting Brownian-type operators with subnormal entry]
{Lifting Brownian-type operators with subnormal entry
   }
   \author[S. Chavan]{Sameer Chavan}
   \address{Department of Mathematics and Statistics\\
Indian Institute of Technology Kanpur, India}
   \email{chavan@iitk.ac.in}
   \author[Z.\ J.\ Jab{\l}o\'nski]{Zenon Jan
Jab{\l}o\'nski}
   \address{Instytut Matematyki,
Uniwersytet Jagiello\'nski, ul.\ \L ojasiewicza 6,
PL-30348 Kra\-k\'ow, Poland}
\email{Zenon.Jablonski@im.uj.edu.pl}
   \author[I.\ B.\ Jung]{Il Bong Jung}
   \address{Department of Mathematics,
Kyungpook National University, Daegu 702-701, Korea}
   \email{ibjung@knu.ac.kr}
   \author[J.\ Stochel]{Jan Stochel}
\address{Instytut Matematyki, Uniwersytet
Jagiello\'nski, ul.\ \L ojasiewicza 6, PL-30348
Kra\-k\'ow, Poland} \email{Jan.Stochel@im.uj.edu.pl}
   \thanks{The research of the second and fourth
authors was supported by the National Science Center
(NCN) Grant OPUS No.\ DEC-2021/43/B/ST1/01651. The
research of the third author was supported by Basic
Science Research Program through the National Research
Foundation of Korea (NRF) funded by the Ministry of
Education (NRF-2021R111A1A01043569).}
   \subjclass[2020]{Primary 47A08, 47A20, 46A22
Secondary 47B20, 47A10}
   \keywords{Upper triangular $2\times 2$ block
matrix, subnormal operator, normal operator, lifting,
spectral inclusion}
   \begin{abstract}
In this paper, we study Brownian-type operators, which
are upper triangular $2\times 2$ block matrix
operators with entries satisfying some algebraic
constraints. We establish a lifting theorem stating
that any Brownian-type operator with subnormal $(2,2)$
entry lifts to a Brownian-type operator with normal
$(2,2)$ entry, where lifting is understood in the
sense of extending entries of the block matrices
representing the operators in question. The spectral
inclusion and the filling in holes theorems are
obtained for such operators.
   \end{abstract}
   \maketitle
   \section{Introduction}
Upper triangular $2\times 2$ block matrices appear in various
parts of operator theory and functional analysis on the
occasion of investigating a variety of topics including the
Halmos similarity problem for polynomially bounded operators
\cite{Fog64,Pi97}, the hyperinvariant subspace problem
\cite{Do-Pe72,Ki12,JKP18,JKP19}, the task of finding models
for the time shift operator for stochastic processes
\cite{Ag-St}, the question of characterizing invertibility of
upper triangular $2\times 2$ block matrices \cite{WLee00},
the task of searching for a model theory for $2$-hyponormal
operators \cite{Cur-Lee02}, the problem of determining a
complete set of unitary invariants for the class of
Cowen-Douglas operators realized as upper triangular block
$2\times 2$ matrices \cite{Mis17}, and many others. This
paper continues the study initiated in \cite{C-J-J-S21} of
Brownian-type operators, which are upper triangular $2\times
2$ block matrix operators with entries satisfying some
algebraic constraints (see Definition~\ref{defq}).

The goal of this paper is to prove a lifting theorem for
Brownian-type operators with subnormal $(2,2)$ entry.
Typically, such a theorem involves finding an extension in a
better-understood class of operators. In our case, such a
natural class consists of Brownian-type operators with normal
$(2,2)$ entry. The main result of the present paper shows
that any Brownian-type operator with subnormal $(2,2)$ entry
lifts to a Brownian-type operator with normal $(2,2)$ entry
(see Theorem~\ref{etrest}), where lifting is understood in
the sense of extending entries of the block matrices
representing the operators in question. Such extensions are
called entrywise extensions (see Definition~\ref{enrty}). In
fact, Theorem~\ref{etrest}(iii) guarantees the existence of
an entrywise extension having additional features related to
the modulus of the $(1,2)$ entry and the minimality of a
normal extension of the $(2,2)$ entry. For such extensions,
we establish the spectral inclusion and the filling in holes
theorems, which usually are attributed to subnormal operators
(see Theorem~\ref{yrwq}). This is rather surprising because
Brownian-type operators with subnormal (or even unitary)
$(2,2)$ entry are generally not subnormal (see
Lemma~\ref{nonur} and Example~\ref{noistx}). Let us mention
that the notion of entrywise extension is essentially
stronger than that of extension. Namely, as shown in
Example~\ref{noistn}, there is a $2$-isometry which is a
Brownian-type operator with subnormal $(2,2)$ entry and which
has no entrywise extension to a Brownian unitary (see
Section~\ref{Sec.4} for the definition). However, by
\cite[Theorem~5.80]{Ag-St} it has an extension to a Brownian
unitary.

Before stating the main theorem and related results,
we provide the necessary concepts. We say that a
bounded linear operator $T$ on a (complex) Hilbert
space $\hh$ is a {\em $2$-isometry} if $T^{*2}T^{2} -
2 T^*T + I =0$. By \cite[Lemma~1]{R-0}, $\triangle_T
:=T^*T-I \Ge 0$ for any $2$-isometry $T$. A
$2$-isometry $T$ is called a {\em Brownian isometry}
if $\triangle_T \triangle_{T^*} \triangle_T=0$, and a
{\em quasi-Brownian isometry} if $\triangle_T T =
\triangle_T^{1/2} T \triangle_T^{1/2}$. The notions of
a $2$-isometry and a Brownian isometry were introduced
in \cite{Ag-0} and \cite{Ag-St}, respectively. An
interest in studying $2$-isometries comes from the
investigation of the invariant subspaces of the
Dirichlet shift \cite{R-0} and the time shift operator
of the modified Brownian motion process \cite{Ag-St}.
The class of quasi-Brownian isometries was studied in
\cite{Maj,B-S,A-C-J-S,A-C-J-S-2}. Let us recall that
non-isometric Brownian and quasi-Brownian isometries
have an upper triangular $2\times 2$ block matrix
representation with entries satisfying some algebraic
constraints (see \cite[Theorem~4.1]{A-C-J-S} and the
references therein). This provided the motivation in
\cite{C-J-J-S21} to study a broader class of operators
called Brownian-type operators, which includes
quasi-Brownian isometries (in particular, Brownian
unitaries and Brownian isometries).

Given two Hilbert spaces $\hh$ and $\kk,$ we denote by
$\ogr{\hh,\kk}$ the Banach space of all bounded linear
operators from $\hh$ to $\kk$. The kernel, the range and the
modulus of an operator $T \in \ogr{\hh,\kk}$ are denoted by
$\jd{T}$, $\ob T$ and $|T|$, respectively. We regard
$\ogr{\hh}:=\ogr{\hh,\hh}$ as a $C^*$-algebra. The identity
operator on $\hh$ is denoted by $I_\hh,$ or simply by $I$ if
no ambiguity arises. For a subset $\mathscr{F}$ of
$\ogr{\hh}$, we denote by $\mathscr{F}'$ the {\em commutant}
of $\mathscr{F}$, which is the set of operators $A\in
\ogr{\hh}$ such that $AB=BA$ for all $B\in \mathscr{F}$. It
is a well-known fact that $\mathscr{F}'$ is a von Neumann
algebra and $\mathscr{F}' = W^*(\mathscr{F})'$, where
$W^*(\mathscr{F})$ denotes the von Neumann algebra generated
by $\mathscr{F} \cup \{I\}$ (see \cite{Mur90} for more
details). Note that if $E\in \ogr{\hh,\kk}$, then it follows
from \cite[Problem~56]{Hal} and the injectivity of
$E|_{\overline{\ob{|E|}}}\colon \overline{\ob{|E|}} \to \kk$
that
   \begin{align} \label{tres}
\dim \overline{\ob{|E|}} \Le \dim \kk.
   \end{align}
We say that $T\in \ogr{\hh,\kk}$ is {\em left-invertible} if
there exists an operator $L\in \ogr{\kk,\hh}$, called a {\em
left-inverse} of $T$, such that $LT=I_{\hh}$. If $T$ is
left-invertible, then $\jd{T}=\{0\}$ and $\ob{T}$ is closed.
It is easy to see that
   \begin{align} \label{luinw}
   \begin{minipage}{63ex}
\textit{$T\in \ogr{\hh,\kk}$ is injective $($resp.,
left-invertible$)$ if and only if $|T|$ is injective
$($resp., invertible$)$}.
   \end{minipage}
   \end{align}
Given two operators $T\in \ogr{\hh}$ and $S\in \ogr{\kk}$, we
write $T\subseteq S$ and say that $S$ {\em extends} $T$, or
that $S$ is an {\em extension} of $T$, if $\hh \subseteq \kk$
(an isometric embedding) and $Th=Sh$ for every~$h\in \hh$.

We say that $S\in \ogr{\hh}$ is {\em subnormal} if there
exist a Hilbert space $\kk$ and a normal operator $N\in
\ogr{\kk}$ such that $S \subseteq N$; if this is the case,
the operator $N$ is called a {\em normal extension} of $S$. A
normal extension $N\in \ogr{\kk}$ of $S$ is said to be {\em
minimal} if $\kk$ is the only closed subspace of $\kk$
containing $\hh$ and reducing $N$; if this is the case, we
write $N=\mathrm{mne}\, S$. Note that if $N$ is a normal
extension of $S$, then $P \in \{N\}'$ if and only if $S$ is
normal, where $P\in \ogr{\kk}$ is the orthogonal projection
of $\kk$ onto $\hh$. Recall that (see
\cite[Proposition~II.2.4]{Co91})
   \begin{align} \label{mynipu}
   \begin{minipage}{70ex}
{\em a normal extension $N\in \ogr{\kk}$ of a subnormal
operator $S\in \ogr{\hh}$ is minimal if and only if $\kk =
\bigvee_{n=0}^{\infty} N^{*n}\hh$}.
   \end{minipage}
   \end{align}
Any subnormal operator has a minimal normal extension. If
$N=\mathrm{mne}\, S$ and $A\in \{S,S^*\}^{\prime}$, then
there exists a unique $B\in \{N\}^{\prime} \cap
\{P\}^{\prime}$ such that $A \subseteq B$, where $P$ is as
above (see \cite[Corollary, p.\ 88]{Br55}). Such $B$ is
called the {\em lift of} $A$ {\em relative to}
$\{N\}^{\prime}$; we also say that $A$ {\em lifts} to
$\{N\}^{\prime}$. The uniqueness of $B$ follows from the
minimality of $N$ and the fact that
   \begin{align} \label{cioka}
\textit{$BN^{*n}h=N^{*n}Ah$ for all integers $n\Ge 0$
and all vectors $h\in \hh$.}
   \end{align}
The following fact is a consequence of the Fuglede
theorem.
   \begin{align} \label{luftr}
   \begin{minipage}{65ex}
{\em If $B$ is the lift of $A$ relative to
$\{N\}^{\prime}$, then $B^*$ is the lift of $A^*$
relative to $\{N\}^{\prime}$.}
   \end{minipage}
   \end{align}
We call an operator $T\in \ogr{\hh}$ {\em hyponormal}
if $\|T^*f\| \Le \|Tf\|$ for all $f \in \hh$. We say
that $T\in \ogr{\hh}$ is {\em quasinormal} if
$TT^*T=T^*TT$, or equivalently if and only if
$T|T|=|T|T$ (see \cite{Brow53}). Quasinormal operators
are subnormal and subnormal operators are hyponormal
(see \cite[Propositions~II.1.7 and II.4.2]{Co91}).
Summarizing, we have the following inclusions:
   \begin{align} \label{minuclus}
   \begin{gathered}
\mathscr{U} \subseteq \mathscr{I} \subseteq
\mathscr{Q},
   \\
\mathscr{U} \subseteq \mathscr{N} \subseteq
\mathscr{Q} \subseteq \mathscr{S} \subseteq
\mathscr{H},
   \end{gathered}
   \end{align}
where $\mathscr{U}$, $\mathscr{I}$, $\mathscr{N}$,
$\mathscr{Q}$, $\mathscr{S}$ and $\mathscr{H}$ stand for the
classes of unitary, isometric, normal, quasinormal, subnormal
and hyponormal operators, respectively. We refer the reader
to \cite{Co91} for more information on the above concepts.

Now we can recall the definition of the Brownian-type
operator.
   \begin{dfn}[cf.\ {\cite[Definition~1.1]{C-J-J-S21}}]
\label{defq}
   We say that $T\in \ogr{\hh}$ is a Brownian-type
operator if it has the block matrix form
   \begin{align*}
T = \begin{bmatrix} V & E \\ 0 & X
\end{bmatrix}
   \end{align*}
with respect to an orthogonal decomposition $\hh=\hh_1
\oplus \hh_2$, where the spaces $\hh_1$ and $\hh_2$
are nonzero and the operators $V\in \ogr{\hh_1},$
$E\in \ogr{\hh_2,\hh_1}$ and $X\in \ogr{\hh_2}$
satisfy the following conditions:
   \allowdisplaybreaks
   \begin{gather} \label{gqb-1}
\text{$V$ is an isometry, i.e., $V^*V=I,$}
   \\  \label{gqb-2}
V^*E=0,
   \\ \label{gqb-3}
XE^*E=E^*EX.
   \end{gather}
If $X$ is unitary $($resp., isometric, normal,
quasinormal, subnormal, hyponormal, etc.$)$, then $T$
is called a Brownian-type operator of class
$\mathscr{U}$ $($resp., $\mathscr{I}$, $\mathscr{N},$
$\mathscr{Q},$ $\mathscr{S},$ $\mathscr{H},$ etc.$)$.
If $\mathscr{X}$ is one of the above classes, we say
simply that $T$ is an operator of class $\mathscr{X}$
and write $T = \big[\begin{smallmatrix} V & E \\
0 & X \end{smallmatrix}\big] \in
\mathscr{X}_{\hh_1,\hh_2}$.
   \end{dfn}
In connection with Definition~\ref{defq}, it is worth
making a few comments.
   \begin{rem}
First, note that by the square root theorem
\cite[Theorem~2.4.4]{Sim-4}, the condition
\eqref{gqb-3} is equivalent to
   \begin{align} \label{gqb-3b}
X|E|=|E|X.
   \end{align}
This yields the following.
   \begin{align} \label{reotp}
   \begin{minipage}{68ex}
{\em If $T=\big[\begin{smallmatrix} V & E \\
0 & X \end{smallmatrix}\big]$ is a Brownian-type
operator, then $\overline{\ob{|E|}}$ reduces $X$.}
   \end{minipage}
   \end{align}
Secondly, note that by \eqref{minuclus}, we have the
following chains of inclusions:
   \begin{align*}
   \begin{gathered}
\mathscr{U}_{\hh_1,\hh_2} \subseteq
\mathscr{I}_{\hh_1,\hh_2} \subseteq
\mathscr{Q}_{\hh_1,\hh_2},
   \\
\mathscr{U}_{\hh_1,\hh_2}
\subseteq\mathscr{N}_{\hh_1,\hh_2} \subseteq
\mathscr{Q}_{\hh_1,\hh_2} \subseteq
\mathscr{S}_{\hh_1,\hh_2} \subseteq
\mathscr{H}_{\hh_1,\hh_2}.
   \end{gathered}
   \end{align*}
Thirdly, one can deduce from \cite[Proposition~5.37 and
Theorem~5.48]{Ag-St} (resp., \cite[Proposition~ 5.1]{Maj})
that a non-isometric operator $T\in \ogr{\hh}$ is a Brownian
isometry (resp., a quasi-Brownian isometry) if and only if
$T$ is of class $\mathscr{U}$ (resp., $\mathscr{I}$); to
avoid the injectivity of $E$, see
\cite[Theorem~4.1]{A-C-J-S}. This means that Brownian
isometries are quasi-Brownian. In view of \cite[Example~
4.4]{A-C-J-S}, the converse implication is not true
in~general.
   \hfill $\diamondsuit$
   \end{rem}
Next, we define the notion of entrywise extension, which
plays a key role in this paper.
   \begin{dfn}  \label{enrty}
Let $\big[\begin{smallmatrix} A_{11} & A_{12} \\
A_{21} & A_{22} \end{smallmatrix}\big]
\in \ogr{\hh_1 \oplus \hh_2}$ and
$\big[\begin{smallmatrix} \tilde A_{11} & \tilde A_{12} \\
\tilde A_{21} & \tilde A_{22}
\end{smallmatrix}\big] \in \ogr{\kk_1
\oplus \kk_2}$. We say that
$\big[\begin{smallmatrix} \tilde A_{11} & \tilde A_{12} \\
\tilde A_{21} & \tilde A_{22}
\end{smallmatrix}\big]$ is an
entrywise extension of $\big[\begin{smallmatrix}
A_{11} & A_{12} \\
A_{21} & A_{22} \end{smallmatrix}\big]$
if
   \begin{align} \label{hikpur}
\text{$\hh_j \subseteq \kk_j$ and
$A_{ij}h_j = \tilde A_{ij}h_j$ for $h_j
\in \hh_j$ and $i,j = 1,2$.}
   \end{align}
If this is the case, then we write
   \begin{align} \label{asdlet}
\begin{bmatrix} A_{11} & A_{12} \\
A_{21} & A_{22}
   \end{bmatrix}
  \preceq
\begin{bmatrix} \tilde A_{11} & \tilde A_{12} \\
\tilde A_{21} & \tilde A_{22}
   \end{bmatrix}.
   \end{align}
   \end{dfn}
Obviously, \eqref{asdlet} implies that
   \begin{align*}
\begin{bmatrix} A_{11} & A_{12} \\
A_{21} & A_{22} \end{bmatrix} \subseteq
\begin{bmatrix} \tilde A_{11} & \tilde A_{12} \\
\tilde A_{21} & \tilde A_{22}
\end{bmatrix}  \quad  \text{(extension of block matrix
operators).}
   \end{align*}
The converse implication holds provided $\hh_j \subseteq
\kk_j$ for $j=1,2$. Note that if the orthogonal
decompositions $\hh_1\oplus \hh_2$ and $\kk_1\oplus \kk_2$
are not related to each other as in \eqref{hikpur}, then
there may exist extensions that are not entrywise (see
Example~\ref{noistn}).

We now state the main result of this paper, which is a
lifting theorem for Brownian-type operators of class
$\mathscr{S}$.
   \begin{thm}\label{etrest}
Let $T=\big[\begin{smallmatrix} V & E \\
0 & S \end{smallmatrix}\big] \in\ogr{\hh_1 \oplus \hh_2}$ and
$\dcal$ be a closed subspace of $\hh_1 \ominus \big(\ob{V}
\vee \overline{\ob{E}}\big)$. Then the following conditions
are equivalent{\em :}
   \begin{enumerate}
   \item[(i)] $T \in \gsbh$,
   \item[(ii)] there
exist Hilbert spaces $\kk_1$ and $\kk_2$ and
$\big[\begin{smallmatrix} \tilde V & \tilde E \\ 0 & N
\end{smallmatrix}\big] \in \gnbhh{\kk_1, \kk_2}$ such
that $T \preceq \big[\begin{smallmatrix} \tilde V &
\tilde E \\ 0 & N \end{smallmatrix}\big]$ and
   \begin{align} \label{pruypec}
|E| \subseteq |\tilde E|,
   \end{align}
   \item[(iii)] for any
minimal normal extension $N\in
\ogr{\kk_2}$ of $S$, there exist a
Hilbert space $\kk_1$ and operators
$\tilde V\in \ogr{\kk_1}$ and $\tilde
E\in \ogr{\kk_2,\kk_1}$ such that
$\big[\begin{smallmatrix} \tilde V &
\tilde E \\ 0 & N \end{smallmatrix}\big]
\in \gnbhh{\kk_1, \kk_2}$, $T\preceq
\big[\begin{smallmatrix} \tilde V &
\tilde E \\ 0 & N \end{smallmatrix}\big]
$, \eqref{pruypec} holds, $\dim \kk_j =
\dim \hh_j$ for $j=1,2$ and
   \begin{align} \label{rozkly}
\dcal = \kk_1 \ominus (\ob{\tilde V}
\oplus \overline{\ob{\tilde E}}).
   \end{align}
   \end{enumerate}
Moreover, if $\big[\begin{smallmatrix} \tilde V &
\tilde E \\ 0 & N
\end{smallmatrix}\big]$ is as in {\em (ii)} and
$N=\mathrm{mne}\,S$, then
   \begin{enumerate}
   \item[(a)] $|\tilde E|$ is the lift of $|E|$ relative to
$\{N\}'$,
   \item[(b)]
$E$ is injective $($resp.,
left-invertible, isometric$)$ if and only
if $\tilde E$ is injective $($resp.,
left-invertible, isometric$)$.
   \end{enumerate}
   \end{thm}
It is worth noting that the condition (ii) of
Theorem~\ref{etrest} in which \eqref{pruypec} is dropped may
not imply (i) (see Remark~\ref{aabbaa} and
Example~\ref{rozwis}). We also refer the reader to
Example~\ref{tryisu} for a discussion of the possible
dimensions of the defect spaces $\hh_1 \ominus \big(\ob{V}
\oplus \overline{\ob{E}}\big)$ and $\kk_1 \ominus (\ob{\tilde
V} \oplus \overline{\ob{\tilde E}})$.
   \begin{cor}
Let $\big[\begin{smallmatrix} V & E \\
0 & S \end{smallmatrix}\big] \in \gsbh$
and let $N\in \ogr{\kk_2}$ be a minimal
normal extension of $S$. Then there exist
a Hilbert space $\kk_1$ and operators
$\tilde V\in \ogr{\kk_1}$ and $\tilde
E\in \ogr{\kk_2,\kk_1}$ such that
$\big[\begin{smallmatrix} \tilde V &
\tilde E \\ 0 & N
\end{smallmatrix}\big] \in \gnbhh{\kk_1,
\kk_2}$,  $\big[\begin{smallmatrix} V & E \\
0 & S \end{smallmatrix}\big] \preceq
\big[\begin{smallmatrix} \tilde V &
\tilde E \\ 0 & N
\end{smallmatrix}\big]$, \eqref{pruypec}
holds, $\dim \kk_j = \dim \hh_j$ for $j=1,2$ and
$\kk_1 = \ob{\tilde V} \oplus \overline{\ob{\tilde
E}}$.
   \end{cor}
The next corollary follows from Theorem~\ref{etrest},
\cite[Theorem~4.1]{A-C-J-S} and \cite[Proposition~6.2]{SF70}.
   \begin{cor}
Any quasi-Brownian isometry has an
entrywise extension to a Brownian
isometry.
   \end{cor}
In the present paper, we are mostly interested in
studying entrywise extensions
$\big[\begin{smallmatrix} \tilde V & \tilde E \\ 0 & N
\end{smallmatrix}\big] \in
\gnbhh{\kk_1, \kk_2}$ of $T=\big[\begin{smallmatrix} V & E \\
0 & S \end{smallmatrix}\big] \in \gsbh$ satisfying
\eqref{pruypec} for which $N=\mathrm{mne}\, S$.
Occasionally, we will call them {\em taut} entrywise
extensions. The next theorem tells us how to model
taut entrywise extensions of $T$ in case the entry $E$
of $T$ is injective.
   \begin{thm} \label{mudyl}
Suppose that $T=\big[\begin{smallmatrix} V & E \\
0 & S \end{smallmatrix}\big] \in\gsbh$, $E$ is
injective and $N\in \ogr{\kk_2}$ is a minimal normal
extension of $S$. Let $E=W|E|$ be the polar
decomposition of $E$. Then $|E|$ is injective, $W$ is
an isometry and the following hold{\em :}
   \begin{enumerate}
   \item[(i)] if  $\big[\begin{smallmatrix} \tilde V & \tilde E
\\ 0 & N \end{smallmatrix}\big] \in \gnbhh{\kk_1,
\kk_2}$ is an entrywise extension of $T$ such that
$|E| \subseteq |\tilde E|$, and $\tilde E=U|\tilde E|$
is the polar decomposition of $\tilde E$, then $\tilde
E$ is injective, $|\tilde E|$ is the lift of $|E|$
relative to $\{N\}'$ and $U\in \ogr{\kk_2,\kk_1}$ is
an isometry such that $W \subseteq U$ and $\ob{\tilde
V} \perp \ob{U}$,
   \item[(ii)] if $\tilde V \in \ogr{\kk_1}$ and $U\in
\ogr{\kk_2,\kk_1}$ are isometries such that
$V\subseteq \tilde V$, $W \subseteq U$ and $\ob{\tilde
V} \perp \ob{U}$, and $B\in \ogr{\kk_2}$ is the lift
of $|E|$ relative to $\{N\}'$, then
$\big[\begin{smallmatrix} \tilde V & \tilde E \\
0 & N \end{smallmatrix}\big] \in \gnbhh{\kk_1, \kk_2}$
with $\tilde E:=UB$, $\big[\begin{smallmatrix} V & E
\\ 0 & S
\end{smallmatrix}\big] \preceq \big[\begin{smallmatrix}
\tilde V & \tilde E \\ 0 & N \end{smallmatrix}\big]$,
$|E| \subseteq |\tilde E|$ and $\tilde E=UB$ is the
polar decomposition of $\tilde{E}$.
   \end{enumerate}
   \end{thm}
The third theorem shows that the second column of any
taut entrywise extension of
$\big[\begin{smallmatrix} V & E \\
0 & S \end{smallmatrix}\big] \in\gsbh$ can always be
decomposed in two parts, one of which has an injective upper
right entry and the other is block-diagonal.
   \begin{thm} \label{mude}
Suppose that  $\big[\begin{smallmatrix} \tilde V & \tilde E \\
0 & N \end{smallmatrix}\big] \in \gnbhh{\kk_1, \kk_2}$
is an entrywise extension of
$T=\big[\begin{smallmatrix} V & E \\
0 & S \end{smallmatrix}\big] \in\gsbh$ satisfying
\eqref{pruypec}, where $N\in \ogr{\kk_2}$ is a minimal
normal extension of $S$. Set $\hh_{21}=\jd{E}$ and
$\hh_{22}=\overline{\ob{|E|}}$. Let
$E_j:=E|_{\hh_{2j}}\colon \hh_{2j}\to \hh_1$ for
$j=1,2$. Then the following statements hold{\em :}
   \begin{enumerate}
   \item[(a)] $\hh_{2j}$ reduces $S$ and
$S_j:=S|_{\hh_{2j}}$ is subnormal for
$j=1,2$,
   \item[(b)] $\hh_2=\hh_{21} \oplus \hh_{22}$ and  $S=S_1
\oplus S_2$,
   \item[(c)] $E(f_1\oplus f_2)=E_1f_1 + E_2f_2$
for $f_j\in \hh_{2j}$, $E_1=0$ and $E_2$
is injective.
   \end{enumerate}
Next, let $N=N_1\oplus N_2$ be the
orthogonal decomposition of $N$ with
respect to an orthogonal decomposition
$\kk_2=\kk_{21}\oplus \kk_{22}$ such that
$N_j=\mathrm{mne}\, S_j$ for $j=1,2$
$($see Lemma~{\em \ref{simurt}}$)$. For
$j=1,2$, let $\tilde E_j:=\tilde
E|_{\kk_{2j}}\colon \kk_{2j} \to \kk_1$.
Then for $j=1,2$,
   \begin{enumerate}
   \item[(d)] $\big[\begin{smallmatrix} V & E_j \\
0 & S_j \end{smallmatrix}\big] \in
\gsbhh{\hh_1, \hh_{2j}}$ and
$\big[\begin{smallmatrix} \tilde V &
\tilde E_j
\\ 0 & N_j \end{smallmatrix}\big] \in
\gnbhh{\kk_1, \kk_{2j}}$,
   \item[(e)]
$\big[\begin{smallmatrix} V & E_j \\
0 & S_j \end{smallmatrix}\big] \preceq
\big[\begin{smallmatrix} \tilde V &
\tilde E_j
\\ 0 & N_j \end{smallmatrix}\big]$,
   \item[(f)] $|\tilde E_j|$ and $|\tilde E|$
are the lifts of $|E_j|$ and $|E|$
relative to $\{N_j\}^{\prime}$ and
$\{N\}^{\prime}$, respectively{\em ;} in
particular, $|E_j| \subseteq |\tilde
E_j|$,
   \item[(g)] $\tilde E_1=0$ and $\tilde E_2$ is injective.
   \end{enumerate}
   \end{thm}
In the following remark, we show what can happen when the
first or the second column of a taut entrywise extension of
$\big[\begin{smallmatrix} V & E \\
0 & S \end{smallmatrix}\big] \in \gsbh$ is extended
entrywise.
   \begin{rem} \label{cysyk}
Suppose $T:=\big[\begin{smallmatrix} V & E \\
0 & S \end{smallmatrix}\big] \in \gsbh$,
$N\in \ogr{\kk_2}$ is a minimal normal
extension of $S$ and $\dcal$ is a closed
subspace of $\hh_1 \ominus \big(\ob{V}
\oplus \overline{\ob{E}}\big)$. Let
$\big[\begin{smallmatrix} \tilde V &
\tilde E \\ 0 & N
\end{smallmatrix}\big] \in \gnbhh{\kk_1,
\kk_2}$ be an entrywise extension of $T$
such that \eqref{pruypec} and
\eqref{rozkly} hold.

First, we consider entrywise extensions relative to
the first column. Take any Hilbert space $\lcal$ and
define the Hilbert space $\mcal_1$ and the operators
$\tilde U \in \ogr{\mcal_1}$ and $\tilde F \in
\ogr{\kk_2,\mcal_1}$ by
   \allowdisplaybreaks
   \begin{align} \label{sum}
   \left.
   \begin{aligned}
\mcal_1&=\kk_1 \oplus \lcal,
   \\
\tilde U (f\oplus g) &= \tilde Vf \oplus Wg, \quad
f\in \kk_1, \, g \in \lcal,
   \\
\tilde F f&= \tilde E f\oplus 0, \quad f\in \kk_2,
   \end{aligned}
   \right\}
   \end{align}
where $W\in \ogr{\lcal}$ is an isometry.
Observing that $|\tilde E|=|\tilde F|$
and $\ob{\tilde E}=\ob{\tilde F}$, it is
easy to see that
$\big[\begin{smallmatrix} \tilde U &
\tilde F \\ 0 & N
\end{smallmatrix}\big] \in \gnbhh{\mcal_1,
\kk_2}$, $T \preceq
\big[\begin{smallmatrix} \tilde V &
\tilde E \\ 0 & N
\end{smallmatrix}\big] \preceq \big[\begin{smallmatrix}
\tilde U & \tilde F \\ 0 & N
\end{smallmatrix}\big]$,  $|\tilde E| = |\tilde F|$ and
   \allowdisplaybreaks
   \begin{align*}
\mcal_1 \ominus (\ob{\tilde U} \oplus
\overline{\ob{\tilde F}}) & = \big[\kk_1 \ominus
(\ob{\tilde V} \oplus \overline{\ob{\tilde E}})\big]
\oplus \big[\lcal \ominus \ob{W}\big]
   \\
&\hspace{-1.4ex}\overset{\eqref{rozkly}}= \dcal \oplus
\jd{W^*}.
   \end{align*}
Since $W$ can be chosen so that $\dim \jd{W^*}$ is of
arbitrary cardinality ($0$ is not excluded), we can
always enlarge entrywise the first column of the
initial extension $\big[\begin{smallmatrix} \tilde V &
\tilde E \\ 0 & N \end{smallmatrix}\big]$ of
$\big[\begin{smallmatrix} V & E \\ 0 & S
\end{smallmatrix}\big]$ still preserving
its properties and guaranteeing that the
dimension of the defect space $\mcal_1
\ominus (\ob{\tilde U} \oplus
\overline{\ob{\tilde F}})$ is of
arbitrary cardinality.

Next, we discuss entrywise extensions
relative to the second column. Our goal
is to show that there is no proper
entrywise extension of
$\big[\begin{smallmatrix} \tilde V &
\tilde E \\ 0 & N
\end{smallmatrix}\big]$ in the class $\mathscr{N}$
preserving the property \eqref{pruypec} and the
minimality of $N$. This can be deduced from the
following observation. If $\mcal_2\subseteq \kk_2$ is
an invariant closed subspace for $N$ such that
$\Big[\begin{smallmatrix} \tilde V & \tilde
E|_{\mcal_2} \\ 0 & N|_{\mcal_2}
\end{smallmatrix}\Big]\in \gnbhh{\kk_1,
\mcal_2}$ and $T\preceq
\Big[\begin{smallmatrix} \tilde V &
\tilde E|_{\mcal_2} \\ 0 & N|_{\mcal_2}
\end{smallmatrix}\Big]$, then
$\mcal_2=\kk_2$. To see that $\mcal_2=\kk_2$, we note that if
$\mcal$ is a closed subspace which is invariant for an
operator $M$ and $M|_{\mcal}$ is normal, then $\mcal$ reduces
$M$. Hence, $\mcal_2$ reduces $N$ and contains $\hh_2$, which
together with the minimality of $N$ yields $\mcal_2=\kk_2$.
   \hfill $\diamondsuit$
   \end{rem}
The organization of the present paper is as follows. The
spectral inclusion and filling in holes theorems for
Brownian-type operators of class $\mathscr{S}$ are given in
Section~\ref{Sec0}. Section~\ref{Sec1} contains preliminary
results needed in subsequent parts of the paper. The proofs
of Theorems~\ref{etrest}, \ref{mudyl} and \ref{mude} will
appear in Section~\ref{Sec2}. In Section~\ref{Sec.4}, we
provide examples illustrating selected issues discussed in
this paper.
   \section{\label{Sec0}Spectral inclusion and filling in holes theorems}
According to Theorem~\ref{etrest}, every operator of
class $\mathscr{S}$ has a taut entrywise extension to
an operator of class $\mathscr{N}$. It turns out that,
as in the case of subnormal operators, the spectral
inclusion and the filling in holes theorems hold for
such extensions (see Theorem~\ref{yrwq}). This is
rather surprising because operators of class
$\mathscr{N}$ are generally not normal (see
Lemma~\ref{nonur}) and, as shown in
Example~\ref{noistx}, there are operators of class
$\mathscr{S}$ (or even of class $\mathscr{U}$) which
are not subnormal. It is worth mentioning that the
question of when an operator of class $\mathscr{Q}$ is
subnormal was resolved in
\cite[Theorem~1.2]{C-J-J-S21}.

For the purposes of this section, we provide the necessary
notations. If $T\in \ogr{\hh}$, then $\sigma(T)$ and
$\rho(T)$ stand for the spectrum and the resolvent set of
$T$, respectively. The set of all complex numbers $\lambda$
such that $|\lambda| \Le 1$ is denoted by $\bar{\dbb}$.

   We begin by characterizing normal operators of
class $\mathscr{N}$.
   \begin{lem} \label{nonur}
If $T=\big[\begin{smallmatrix} V & E \\
0 & N \end{smallmatrix}\big] \in\gnbh$, then $T$ is
normal if and only if $E=0$ and $V$ is unitary.
   \end{lem}
   \begin{proof}  Since   $N$ is normal, we have
   \begin{align*}
T^*T-TT^* = \begin{bmatrix} I-VV^*-EE^* & -EN^* \\
- NE^* & E^*E
   \end{bmatrix}.
   \end{align*}
Hence, $T$ is normal if and only if $E=0$ and
$VV^*=I$.
   \end{proof}
Next, we discuss the spectrum of upper triangular
$2\times 2$ block matrix operators.
   \begin{pro}
Let $T=\big[\begin{smallmatrix} A & B \\
0 & C \end{smallmatrix}\big] \in \ogr{\hh_1 \oplus
\hh_2}$. Then the following statements are valid{\em
:}
   \begin{enumerate}
   \item[(i)] if $C$  is invertible, then $T$ is
invertible if and only if $A$ is invertible,
   \item[(ii)] $\sigma(T) \subseteq \sigma(A) \cup
\sigma(C)$,
   \item[(iii)] if $C$ is hyponormal, then $\sigma(T) = \sigma(A) \cup
\sigma(C)$.
   \end{enumerate}
   \end{pro}
   \begin{proof}
(i) The proof can be found in \cite[p.\ 219]{Hal}.

(ii) This statement follows from (i).

(iii) Suppose that $C$ is hyponormal. In view of (ii),
it is enough to show that $\rho(T) \subseteq \rho(A)
\cap \rho(C)$. Take $\lambda \in \rho(T)$. Since
$T-\lambda I$ is right-invertible, it is easy to see
that $C-\lambda I$ is right-invertible. Knowing that
$C-\lambda I$ is hyponormal (see
\cite[Proposition~II.4.4(b)]{Co91}), we infer from
\cite[Proposition~II.4.10(a)]{Co91} that $C-\lambda I$
is invertible. However $T-\lambda I =
\big[\begin{smallmatrix} A -\lambda I & B \\
0 & C-\lambda I \end{smallmatrix}\big]$, so by (i),
$A-\lambda I$ is invertible. This completes the proof.
   \end{proof}
   \begin{cor} \label{xeq}
If $T=\big[\begin{smallmatrix} V & E \\
0 & S \end{smallmatrix}\big] \in\ghbh$, then
$\sigma(T) = \sigma(V) \cup \sigma(S)$.
   \end{cor}
As shown below, the spectral inclusion and the filling
in holes theorems, which hold for subnormal operators,
are also valid for operators of class $\mathscr{S}$.
   \begin{thm} \label{yrwq}
Suppose that $T=\big[\begin{smallmatrix} V & E \\
0 & S \end{smallmatrix}\big] \in\gsbh$. Let
$R=\big[\begin{smallmatrix} \tilde V & \tilde E \\
0 & N \end{smallmatrix}\big] \in \gnbhh{\kk_1, \kk_2}$
be an entrywise extension of $T$ such that $|E|
\subseteq |\tilde E|$ and $N=\mathrm{mne}\, S$. Then
the following statements are valid{\em :}
   \begin{enumerate}
   \item[(i)] if $E\neq 0$, then $\sigma(T) =
\bar{\dbb} \cup \sigma(S)$ and $\sigma(R) = \bar{\dbb}
\cup \sigma(N)$,
   \item[(ii)] if $V$ is not unitary, then $\sigma(R) \subseteq \sigma(T)$,
   \item[(iii)] if $\varOmega$ is a connected component of
$\rho(R)$ such that $\varOmega \cap \sigma(T) \neq
\varnothing$, then $\varOmega \subseteq \sigma(T)$.
   \end{enumerate}
   \end{thm}
   \begin{proof}
(i) Since $E\neq 0$, we deduce from
Theorem~\ref{etrest}(a) that $\tilde E\neq 0$.
Applying the orthogonality property \eqref{gqb-2} to
$T$ and $R$, we see that $V$ and $\tilde V$ are
non-unitary isometries. Consequently,
$\sigma(V)=\sigma(\tilde V)=\bar{\dbb}$ (see \cite[p.\
213, Exercise~7]{Co90}). Combined with
\eqref{minuclus} and Corollary~\ref{xeq}, this proves
(i).

(ii) It follows from \cite[Theorem~II.2.11(a)]{Co91}
that
   \begin{align} \label{spinc}
\sigma(N) \subseteq \sigma(S).
   \end{align}
By assumption, $V \subseteq \tilde V$. Since $V$ is a
non-unitary isometry and $\|\tilde V\|=1$, we get
   \begin{align*}
\sigma(\tilde V) \subseteq \bar\dbb =\sigma(V).
   \end{align*}
This, together with \eqref{minuclus}, \eqref{spinc}
and Corollary~\ref{xeq}, implies that $\sigma(R)
\subseteq \sigma(T)$.

(iii) Since $T \subseteq R$, the statement (iii) is a
direct consequence of \cite[Theorem]{Dr-Ya20} (see
also \cite[Theorem~4]{Scr59}).
   \end{proof}
We conclude this section with some comments on
Theorem~\ref{yrwq}.
   \begin{rem}
First, it is obvious that the statement (i) without
$E\neq 0$ is not true in general. Indeed, if $E=0$,
then by Theorem~\ref{etrest}(a), $\tilde E=0$, so
$T=V\oplus S$ and $R=\tilde V \oplus N$. Choosing
appropriate isometries $V$ and $\tilde V$, we will
obtain the required counterexamples.

Secondly, the statement (ii) is optimal in the sense
that there are operators $T$ and $R$ satisfying the
assumptions of Theorem~\ref{yrwq} such that $V$ is
unitary and the spectral inclusion $\sigma(R)
\subseteq \sigma(T)$ is not valid. Indeed, if $V$ is
unitary, then $E$ and consequently $\tilde E$ are zero
operators, which brings the issue of calculating
spectra to the case of orthogonal sums as above.
   \hfill $\diamondsuit$
   \end{rem}
   \section{\label{Sec1}Preparatory lemmas}
We begin by showing that the operation of taking
powers with positive real exponents preserves the
inclusion of positive operators.
   \begin{lem} \label{inrx}
Let $\hh$ be a closed subspace of a Hilbert space
$\kk$, $A\in \ogr{\hh}$ and $B\in \ogr{\kk}$ be two
positive operators and $\alpha$ be a positive real
number. Then
   \begin{align*}
A \subseteq B \iff A^{\alpha} \subseteq B^{\alpha}.
   \end{align*}
   \end{lem}
   \begin{proof}
It suffices to prove the ``only if'' part. Suppose
that $A \subseteq B$. Then
   \begin{align} \label{trsz}
\textit{$p(A)\subseteq p(B)$ for every complex
polynomial $p$ in one variable.}
   \end{align}
Let $J$ be a closed finite subinterval of $[0,\infty)$
which contains the spectra of $A$ and $B$. By the
Weierstrass theorem, there exists a sequence
$\{p_n\}_{n=1}^{\infty}$ of complex polynomials in one
variable such that $\lim_{n\to\infty}\sup_{x\in J}
|x^{\alpha}-p_n(x)|=0$. Applying the functional
calculus for selfadjoint operators, we deduce that
$\lim_{n\to \infty}\|A^{\alpha} - p_n(A)\|=0$ and
$\lim_{n\to \infty}\|B^{\alpha} - p_n(B)\|=0$. This
together with \eqref{trsz} implies that $A^{\alpha}
\subseteq B^{\alpha}$.
   \end{proof}
Observe that the same proof as above shows that if $A
\subseteq B$ and $f$ is a continuous complex function
defined on the union of spectra of $A$ and $B$, then
$f(A) \subseteq f(B)$.

Next, we prove that the orthogonal projections onto
kernels of normal elements of von Neumann algebras are
preserved by $*$-isomor\-phisms.
   \begin{lem} \label{inyec}
Let $\acal$ and $\bcal$ be two von Neumann algebras on
Hilbert spaces $\hh$ and $\kk$, respectively, such
that $I_{\hh}\in \acal$ and $I_{\kk}\in \bcal$.
Suppose that $\psi\colon \acal \to \bcal$ is a
$*$-isomorphism and $A\in \acal$ is a normal operator.
Denote by $P_{\jd{A}}$ and $P_{\jd{\psi(A)}}$ the
orthogonal projections of $\hh$ and $\kk$ onto
$\jd{A}$ and $\jd{\psi(A)}$, respectively. Then
   \begin{enumerate}
   \item[(i)] $P_{\jd{A}} \in \acal$ and $P_{\jd{\psi(A)}} \in
   \bcal$,
   \item[(ii)] $\psi(P_{\jd{A}})=P_{\jd{\psi(A)}}$.
   \end{enumerate}
In particular, $A$ is injective if and only if
$\psi(A)$ is injective.
   \end{lem}
   \begin{proof}
(i) Since $P_{\jd{A}}=G(\{0\})$, where $G$ is the
spectral measure of $A$ (see
\cite[Theorem~6.6.2]{Bi-So87}), we infer from
\cite[Theorem~4.1.11(2)]{Mur90} that $P_{\jd{A}} \in
\acal$. The same argument can be applied to $\psi(A)$,
which is normal as the image of a normal operator via
$*$-homomorphism. Thus (i) holds.

(ii) Since $\psi$ is a $*$-homomorphism, we see that
$\psi(P_{\jd{A}})$ is an orthogonal projection in
$\bcal$. Noting that $AP_{\jd{A}}=0$, we deduce that
$\psi(A)\psi(P_{\jd{A}})=0$, so $\ob{\psi(P_{\jd{A}})}
\subseteq \jd{\psi(A)}$, which implies that
   \begin{align} \label{cusyk}
\psi(P_{\jd{A}}) \Le P_{\jd{\psi(A)}}.
   \end{align}
Applying \eqref{cusyk} to the pair $(\psi^{-1},\psi(A))$ in
place of $(\psi,A)$, we obtain
   \begin{align*}
\psi^{-1}(P_{\jd{\psi(A)}}) \Le P_{\jd{\psi^{-1}(\psi(A))}} =
P_{\jd{A}}.
   \end{align*}
Since $\psi$ is $*$-homomorphism, $\psi$ preserves the
order\footnote{Observe that if $A\in \acal$ is such that $A$
is positive in the sense of quadratic forms, then $A$ is a
positive element of $\acal$ in the sense that $A=B^*B$ for
some $B\in \acal$, and {\em vice versa}.} ``$\Le$'', so
$P_{\jd{\psi(A)}} \Le \psi(P_{\jd{A}})$, which together with
\eqref{cusyk} yields (ii). The remaining part of the
conclusion follows from (ii). This completes the proof.
   \end{proof}
The uniqueness of lifts stated in Lemma~\ref{judnluf}
below is a consequence of \eqref{cioka} and
\cite[Proposition~II.2.5]{Co91}. We leave the simple
proof to the reader.
   \begin{lem} \label{judnluf}
Let $N_1\in \ogr{\kk_1}$ and $N_2\in \ogr{\kk_2}$ be
two minimal normal extensions of a subnormal operator
$S\in \ogr{\hh}$. Then there exists a unique unitary
operator $U\in \ogr{\kk_1, \kk_2}$ such that
$U|_{\hh}=I_{\hh}$ and $UN_1=N_2U$. Moreover, if
$B_1\in \ogr{\kk_1}$ and $B_2\in \ogr{\kk_2}$ are the
lifts of $A\in \{S,S^*\}^{\prime}$ relative to
$\{N_1\}'$ and $\{N_2\}'$, respectively, then
$UB_1=B_2U$.
   \end{lem}
Our next goal is to show that the lift of the
orthogonal sum of two operators is the orthogonal sum
of the lifts of components.
   \begin{lem} \label{simurt}
Let $S_1 \in \ogr{\hh_1}$ and $S_2 \in \ogr{\hh_2}$ be
subnormal operators and let $N\in \ogr{\kk}$ be a
minimal normal extension of $S:=S_1\oplus S_2$. Then
   \begin{enumerate}
   \item[(i)] $N$ can be $($uniquely$)$ decomposed as
$N=N_1 \oplus N_2$ relative to an orthogonal
decomposition $\kk=\kk_1 \oplus \kk_2$ in such a way
that $N_j=\mathrm{mne}\, S_j$ for $j=1,2$,
   \item[(ii)] if $A\in \{S,S^*\}^{\prime}$ is
of the form $A=A_1\oplus A_2$ relative to $\hh_1\oplus
\hh_2$ and $B$ is the lift of $A$ relative to
$\{N\}'$, then $\kk_j$ reduces $B$ and $B|_{\kk_j}$ is
the lift of $A_j$ relative to $\{N_j\}'$ for $j=1,2$.
   \end{enumerate}
   \end{lem}
   \begin{proof}
(i) Fix $j\in \{1,2\}$. Set $\kk_j
=\bigvee_{n=0}^{\infty} N^{*n}\hh_j$. It is easy to
see that $\kk_j$ reduces $N$ and $N_j:=N|_{\kk_j}$ is
a normal extension of $S_j$. Since
   \begin{align} \label{grqa}
\kk_j=\bigvee_{n=0}^{\infty} N^{*n}\hh_j =
\bigvee_{n=0}^{\infty} N_j^{*n}\hh_j,
   \end{align}
we infer from \eqref{mynipu} that $N_j$ is a minimal normal
extension of $S_j$. Since for all $f\in \hh_1$, $g \in \hh_2$
and $m,n \in \{0,1,2, \ldots\}$,
   \begin{align*}
\is{N^{*m}f}{N^{*n}g}= \is{N^n f}{N^{m}g} =\is{S_1^n
f}{S_2^{m}g}=0,
   \end{align*}
we conclude that the spaces $\kk_1$ and $\kk_2$ are
orthogonal. Combined with \eqref{mynipu}, this implies that
\allowdisplaybreaks
   \begin{align*}
\kk =\bigvee_{n=0}^{\infty} N^{*n}(\hh_1\oplus \hh_2)
= \bigvee_{n=0}^{\infty} N_1^{*n}\hh_1 \oplus
\bigvee_{n=0}^{\infty} N_2^{*n}\hh_2
\overset{\eqref{grqa}}= \kk_1 \oplus \kk_2.
   \end{align*}
Hence $N=N_1\oplus N_2$. The proof of the uniqueness goes as
follows. If $N= M_1\oplus M_2$ is an another decomposition of
$N$ relative to $\kk=\mcal_1\oplus \mcal_2$ and
$M_j=\mathrm{mne}\, S_j$, then by \eqref{mynipu} we have
   \begin{align*}
\mcal_j=\bigvee_{n=0}^{\infty} (M_j)^{*n}\hh_j =
\bigvee_{n=0}^{\infty} N^{*n}\hh_j
\overset{\eqref{grqa}}= \kk_j.
   \end{align*}

(ii) Fix $j\in \{1,2\}$. Clearly, $A_j\in
\{S_j,S_j^*\}^{\prime}$. Since
   \begin{align*}
B(N^{*n}h_j) \overset{\eqref{cioka}}= N^{*n} A h_j =
N^{*n} A_jh_j, \quad n\Ge 0, \, h_j\in \hh_j,
   \end{align*}
we infer from \eqref{grqa} that $\kk_j$ is invariant
for $B$. Applying this to $A^*$ and $B^*$ (see
\eqref{luftr}) shows that $\kk_j$ is invariant for
$B^*$, so $\kk_j$ reduces $B$. Hence $B|_{\kk_j} \in
\{N_j\}'$ and $A_j \subseteq B|_{\kk_j}$. Applying
this to $A^*$ and $B^*$ again yields $A_j^* \subseteq
(B|_{\kk_j})^*$. Therefore $B|_{\kk_j} \in \{N_j\}'
\cap \{P_j\}'$, where $P_j\in \ogr{\kk_j}$ is the
orthogonal projection of $\kk_j$ onto $\hh_j$.
Consequently, $B|_{\kk_j}$ is the lift of $A_j$
relative to $\{N_j\}'$.
   \end{proof}
The question of cohyponormality of upper triangular
$2\times 2$ block matrix operators is discussed below.
   \begin{lem} \label{nurm}
Suppose that $T=\big[\begin{smallmatrix} A & X \\
0 & Y \end{smallmatrix}\big] \in \ogr{\hh_1 \oplus \hh_2}$
and $Y$ is hyponormal. Then the following conditions are
equivalent{\em :}
   \begin{enumerate}
   \item[(i)] $T^*$ is hyponormal,
   \item[(ii)] $X=0$, $A^*$ is
hyponormal and $Y^*$ is hyponormal,
   \item[(iii)] $X=0$, $A^*$  is
hyponormal and $Y$ is normal.
   \end{enumerate}
In particular, if $T$ is a Brownian-type operator of
class $\mathscr{H}$ and $T^*$ is hyponormal, then
$X=0$ and $A$ is unitary.
   \end{lem}
   \begin{proof}  Writing $[C,D]=CD-DC$ for $C,D \in
\ogr{\hh}$, observe that
   \begin{align} \label{iliema}
[T,T^*] = \begin{bmatrix} XX^* - [A^*,A] & XY^* - A^*X
\\[2ex]
YX^* - X^*A & -(X^*X + [Y^*,Y])
\end{bmatrix}.
   \end{align}

(i)$\Rightarrow$(iii) Since $[T,T^*] \Ge 0$, we deduce
from \eqref{iliema} that $X^*X + [Y^*,Y] \Le 0$, so
$X^*X \Le -[Y^*,Y] \Le 0$, which implies that $X=0$
and $[Y^*,Y]=0$. Combined with \eqref{iliema}, this
shows that $[A^*,A]\Le 0$, which yields (iii).

(iii)$\Rightarrow$(i) This is a direct consequence of
\eqref{iliema}.

(ii)$\Leftrightarrow$(iii) This is obvious.

Finally, the ``in particular'' part follows from the
implication (i)$\Rightarrow$(ii).
   \end{proof}
   \begin{cor} \label{minfik}
Let $N\in \ogr{\kk}$ be a normal extension of a
subnormal operator $S\in \ogr{\hh}$. Then
   \begin{enumerate}
   \item[(i)] if
$\big[\begin{smallmatrix} S & X \\
0 & Y \end{smallmatrix}\big]$ is the $2 \times 2$
block matrix representation of $N$ with respect to the
orthogonal decomposition $\kk = \hh \oplus
\hh^{\perp}$ and $Y$ is hyponormal, then $X=0$,
   \item[(ii)] if $\dim \hh^{\perp} < \infty$,
then $S$ is normal and $\hh$ reduces $N${\em ;} in
particular, if $\dim \hh^{\perp} < \infty$ and
$N=\mathrm{mne}\, S$, then $S=N$.
   \end{enumerate}
   \end{cor}
   \begin{proof}
(i) This a direct consequence of Lemma~\ref{nurm}.

(ii) It is well-known that $Y^*$ is a subnormal
operator with normal extension $N^*$ (see \cite[p.\
40]{Co91}). Since $Y^*$ acts on a finite-dimensional
space, $Y$ is normal, so applying (i) completes the
proof.
   \end{proof}
Powers of Brownian-type operators can be described as
follows (cf.\ \cite[Proposition~3.10]{C-J-J-S21}).
   \begin{lem} \label{xyzzyx}
Let $T = \big[\begin{smallmatrix} V & E \\
0 & S \end{smallmatrix}\big] \in \ogr{\hh_1\oplus \hh_2}$ be
a Brownian-type operator and
   \begin{align} \label{syru}
E_n :=
   \begin{cases}
0 & \text{ for } n=0,
   \\
\sum_{j=0}^{n-1} V^jES^{n-1-j} & \text{ for } n\Ge 1.
   \end{cases}
   \end{align}
Then the following statements are valid{\em :}
   \begin{enumerate}
   \item[(i)] $T^n= \big[\begin{smallmatrix} V^n & E_n \\
0 & S^n \end{smallmatrix}\big]$ for $n\Ge 0$,
   \item[(ii)] $E_n^*E_n =
E^*E\sum_{j=0}^{n-1}S^{*j}S^j=\big(\sum_{j=0}^{n-1}S^{*j}S^j
\big) E^*E$ for $n\Ge 1$,
   \item[(iii)] $E^*E$ commutes with
$S^{*i}S^j$ for all $i,j\Ge 0$.
   \end{enumerate}
   \end{lem}
   \begin{proof}
The statement (iii) follows from \eqref{gqb-3} applied
to $X=S$. To justify (i) and (ii), we can argue as in
the proof of \cite[Proposition~3.10(i)]{C-J-J-S21}
using~(iii).
   \end{proof}
Now, we give a necessary and sufficient condition for
the $n$th power of a Brownian-type operator to be of
class~$\mathscr{S}$.
   \begin{pro} \label{xyzzyv}
Suppose that $T = \big[\begin{smallmatrix} V & E \\
0 & S \end{smallmatrix}\big] \in \ogr{\hh_1 \oplus \hh_2}$ is
a Brownian-type operator. Then the space
$\hh_{22}:=\overline{\ob{|E|}}$ reduces $S$ and for any $n\Ge
2$, the following conditions are equivalent{\em :}
   \begin{enumerate}
   \item[(i)] $T^n \in \gsbh$,
   \item[(ii)] $S^n$ is subnormal and $S_2^n$
commutes with $\sum_{j=1}^{n-1}S_2^{*j}S_2^j$, where
$S_2:=S|_{\hh_{22}}$.
   \end{enumerate}
In particular, if $E$ is injective, then for any $n\Ge
2$, $T^n \in \gsbh$ if and only if $S^n$ is subnormal
and $S^n$ commutes with $\sum_{j=1}^{n-1}S^{*j}S^j$.
   \end{pro}
   \begin{proof}
First, we show that the conditions (i) and (ii) are
equivalent. Using \eqref{gqb-2} and the fact that $V$ is an
isometry, we see that $V^{*n}E_n=0$ for any $n\Ge 0$, where
$E_n$ is as in \eqref{syru}. This and Lemma~\ref{xyzzyx}(i)
imply that for every $n\Ge 2$, (i) holds if and only if $S^n$
is subnormal and $S^n$ commutes with $E_n^*E_n$. Fix $n\Ge
2$. Note that by \eqref{reotp}, the space $\hh_{22}$ reduces
$S$. Thus, the following equalities~hold:
   \allowdisplaybreaks
   \begin{align*}
   \left.
   \begin{aligned}
S_2^n \bigg(\sum_{j=0}^{n-1} S_2^{*j}S_2^j\bigg)
(E^*Eh) & = S^n
\bigg(\sum_{j=0}^{n-1}S^{*j}S^j\bigg)E^*Eh
   \\
\bigg(\sum_{j=0}^{n-1} S_2^{*j}S_2^j\bigg) S_2^n (E^*Eh) &
\overset{\eqref{gqb-3}}=
\bigg(\sum_{j=0}^{n-1}S^{*j}S^j\bigg) E^*E S^n h
   \end{aligned}
   \right\}, \quad h\in \hh_2.
   \end{align*}
Combined with the statement (ii) of Lemma~\ref{xyzzyx} and
the equality $\hh_{22}=\overline{\ob{E^*E}}$, this implies
that $S^n$ commutes with $E_n^*E_n$ if and only if (ii)
holds.

``In particular'' part follows from \eqref{luinw} and
the equivalence (i)$\Leftrightarrow$(ii).
   \end{proof}
The following corollary is a direct consequence of
Proposition~\ref{xyzzyv} and the square root theorem.
   \begin{cor}
Let  $T = \big[\begin{smallmatrix} V & E \\
0 & S \end{smallmatrix}\big] \in \gsbh$. Then $T^2 \in
\gsbh$ if and only if the operators $S_2^2$ and
$|S_2|$ commute, where
$S_2:=S|_{\overline{\ob{|E|}}}$. In particular, if $E$
is injective, then $T^2 \in \gsbh$ if and only if the
operators $S^2$ and $|S|$ commute.
   \end{cor}
A part of the next corollary appeared in
\cite[Proposition~3.10(i)]{C-J-J-S21}.
   \begin{cor}
\label{kluskry} Any positive integer power of an
operator of class $\mathscr{U}$ $($resp.,
$\mathscr{I}$, $\mathscr{N}$, $\mathscr{Q}$$)$ is of
class $\mathscr{U}$ $($resp., $\mathscr{I}$,
$\mathscr{N}$, $\mathscr{Q}$$)$.
   \end{cor}
   \begin{proof}
The case of operators of class $\mathscr{Q}$ follows from
Lemma ~\ref{xyzzyx}(i) and Proposition~\ref{xyzzyv}, and the
fact that for any quasinormal operator $Q\in \ogr{\hh}$ and
for every integer $n \Ge 1$, $Q^n$ is quasinormal and
$Q^{*n}Q^n=(Q^*Q)^n$ (see e.g., \cite[Corollary~3.8 and
Theorem~3.6]{J-J-S14}). Since $\mathscr{U}, \mathscr{I},
\mathscr{N} \subseteq \mathscr{Q}$ and each of these classes
is closed under the operation of taking positive integer
powers, an application of the quasinormal case along with
Lemma~\ref{xyzzyx}(i) completes the proof.
   \end{proof}
The parallel issue of entrywise extensions of powers
of operators of class $\mathscr{S}$ to operators of
class $\mathscr{N}$ is discussed below.
   \begin{pro}\label{ahj}
Let $T = \big[\begin{smallmatrix} V & E \\
0 & S \end{smallmatrix}\big] \in \gsbh$,
$R=\big[\begin{smallmatrix} \tilde V & \tilde E \\ 0 &
N
\end{smallmatrix}\big]\in
\gnbhh{\kk_{1},\kk_{2}}$ and $T \preceq R$. Fix an
integer $n\Ge 1$. Let $E_n$ be as in \eqref{syru} and
$\tilde E_n = \sum_{j=0}^{n-1} \tilde V^j\tilde E
N^{n-1-j}$.~Then
   \begin{enumerate}
   \item[(i)] $\big[\begin{smallmatrix}
V^n & E_n \\ 0 & S^n
\end{smallmatrix}\big] = T^n \preceq R^n
= \big[\begin{smallmatrix}
\tilde V^n &  \tilde E_n \\
0 & N^n \end{smallmatrix}\big] \in
\gnbhh{\kk_{1},\kk_{2}}$,
   \item[(ii)] if $|E| \subseteq
|\tilde E|$, then $E_n^*E_n = P \tilde E_n^*\tilde E_n
|_{\hh_2}$, where $P\in \ogr{\kk_2}$ stands for the
orthogonal projection of $\kk_2$ onto $\hh_2$,
   \item[(iii)] if $N = \mathrm{mne}\,
S$, then $N^n = \mathrm{mne}\, S^n$,
   \item[(iv)] if $T\in \gqbh$ and $|E|\subseteq |\tilde
E|$, then $T^n \in \gqbh$, $T^n \preceq R^n \in
\gnbhh{\kk_{1},\kk_{2}}$ and $|E_n| \subseteq |\tilde
E_n|$,
   \item[(v)]
if $T\in \gibh$ and $N = \mathrm{mne}\, S$, then
$R^n\in \gubhh{\kk_1,\kk_2}$ and $N^n = \mathrm{mne}\,
S^n$.
   \end{enumerate}
   \end{pro}
   \begin{proof}
(i) This follows from Lemma~\ref{xyzzyx},
Corollary~\ref{kluskry} and \eqref{mulmnoz}.
   \begin{align} \label{mulmnoz}
  \left.
\begin{minipage}{66ex} {\em If
\eqref{asdlet}
holds,  $\big[\begin{smallmatrix} B_{11} & B_{12} \\
B_{21} & B_{22} \end{smallmatrix}\big] \in \ogr{\hh_1
\oplus \hh_2}$, $\big[\begin{smallmatrix}
\tilde B_{11} & \tilde B_{12} \\
\tilde B_{21} & \tilde B_{22}
\end{smallmatrix}\big] \in \ogr{\kk_1 \oplus \kk_2}$ and
$\big[\begin{smallmatrix} B_{11} & B_{12} \\
B_{21} & B_{22} \end{smallmatrix}\big] \preceq
\big[\begin{smallmatrix}
\tilde B_{11} & \tilde B_{12} \\
\tilde B_{21} & \tilde B_{22}
\end{smallmatrix}\big]$, then}
$$
\begin{bmatrix} A_{11} & A_{12} \\
A_{21} & A_{22} \end{bmatrix} \begin{bmatrix} B_{11} & B_{12} \\
B_{21} & B_{22} \end{bmatrix} \preceq
\begin{bmatrix}
\tilde A_{11} & \tilde A_{12} \\
\tilde A_{21} & \tilde A_{22}
\end{bmatrix}
\begin{bmatrix} \tilde B_{11} & \tilde B_{12} \\
\tilde B_{21} & \tilde B_{22}
\end{bmatrix}.
   $$
   \end{minipage}
   \right\}
   \end{align}

(ii) Applying Lemma~\ref{xyzzyx}(ii) to $T$ and $R$ and using
the inclusion $|E| \subseteq |\tilde E|$~yields
\allowdisplaybreaks
   \begin{align*}
\is{P\tilde E_n^* \tilde E_n|_{\hh_2} f}{f} & =
\Big\langle\sum_{j=0}^{n-1}N^{*j}N^j |\tilde E|^2 f,
f\Big \rangle
   \\
&\hspace{-.3ex}\overset{\eqref{gqb-3b}}=
\sum_{j=0}^{n-1}\|N^j |\tilde E|f\|^2 =
\sum_{j=0}^{n-1}\|S^j |E|f\|^2
   \\
&\hspace{-.3ex}\overset{\eqref{gqb-3b}} =
\Big\langle\sum_{j=0}^{n-1}S^{*j}S^j |E|^2 f, f
\Big\rangle = \is{E_n^* E_n f}{f}, \quad f \in \hh_2.
   \end{align*}
This proves (ii).

(iii) This follows from \cite[Theorem~4.1]{Ol76} (the
assumption $n\Ge 1$ is essential).

(iv) By Corollary~\ref{kluskry} and (i), $T^n \in \gqbh$ and
$T^n \preceq R^n \in \gnbhh{\kk_{1},\kk_{2}}$. It remains to
show that $|E_n| \subseteq |\tilde E_n|$. Note that since
$S=N|_{\hh_2}$ is quasinormal, \cite[Theorem~1]{Cam75}
implies that $\hh_2$ is invariant for $N^*N$ (see also
\cite[Lemma~1(ii)]{Em81}). It follows from the inclusion $|E|
\subseteq |\tilde E|$ and Lemma~\ref{xyzzyx}(ii) that
   \begin{align*}
\tilde E_n^*\tilde E_nf = \sum_{j=0}^{n-1}(N^*N)^j
|E|^2 f, \quad f\in \hh_2,
   \end{align*}
which implies that $\hh_2$ is invariant for $\tilde
E_n^*\tilde E_n$. Combined with (ii) and
Lemma~\ref{inrx}, this implies that $|E_n| \subseteq
|\tilde E_n|$.

(v) This is a consequence of (i), (iii) and the fact
that a minimal normal extension of an isometric
operator is unitary. This completes the proof.
   \end{proof}
Applying the statements (iii) and (iv) of
Proposition~\ref{ahj}, we can describe entrywise
extensions of powers of operators of class
$\mathscr{Q}$ to operators of class $\mathscr{N}$.
   \begin{cor}
Let $T = \big[\begin{smallmatrix} V & E \\
0 & S \end{smallmatrix}\big] \in \gqbh$,
$R=\big[\begin{smallmatrix} \tilde V & \tilde E \\ 0 &
N \end{smallmatrix}\big]\in \gnbhh{\kk_{1},\kk_{2}}$
and $T \preceq R$. Suppose that $|E|\subseteq |\tilde
E|$. Fix an integer $n\Ge 1$. Let $E_n$ and $\tilde
E_n$ be as in Proposi\-tion~{\em \ref{ahj}}. Then $T^n
\in \gqbh$, $\big[\begin{smallmatrix} V^n & E_n
\\ 0 & S^n
\end{smallmatrix}\big] = T^n \preceq R^n
= \big[\begin{smallmatrix}
\tilde V^n &  \tilde E_n \\
0 & N^n \end{smallmatrix}\big] \in
\gnbhh{\kk_{1},\kk_{2}}$ and $|E_n| \subseteq |\tilde
E_n|$. If, moreover, $N = \mathrm{mne}\, S$, then $N^n
= \mathrm{mne}\, S^n$.
   \end{cor}
   \begin{rem} \label{aabbaa}
In view of Proposition~\ref{ahj} (and under its assumptions),
for every integer $n\Ge 1$, $T^n$ has an entrywise extension
to an operator of class $\mathscr{N}$. However, it is not
true, in general, that $T^n \in \gsbh$ for all integers $n\Ge
2$. In fact, as shown in Example~\ref{rozwis}, there is an
operator $T$ of class $\mathscr{S}$ such that $T^n$ is not of
class $\mathscr{S}$ for all $n\Ge 2$. For such $T$, $|E_n|
\not\subseteq |\tilde E_n|$ for all $n\Ge 2$. Indeed, if this
were not true, then by virtue of Corollary~\ref{kluskry} and
Theorem~\ref{etrest}, $T^n$ would be of class $\mathscr{S}$
for some $n\Ge 2$, which is not the case. This shows that the
condition (ii) of Theorem~\ref{etrest} without
\eqref{pruypec} may not imply (i).
   \hfill $\diamondsuit$
   \end{rem}
   We now discuss some circumstances in which at least
one of the spaces $\hh_1$ and $\hh_2$ appearing in
Definition~\ref{defq} is finite dimensional.
   \begin{lem}\label{skwet}
Let $T=\big[\begin{smallmatrix} V & E \\
0 & S \end{smallmatrix}\big] \in \ogr{\hh_1 \oplus \hh_2}$ be
a Brownian-type operator. Then
   \begin{enumerate}
   \item[(i)] if $\dim \hh_1 < \infty$,
then $E=0$ and $V$ is unitary,
   \item[(ii)] if $T \in \ghbh$ and $\dim
\hh_2 < \infty$,
then $\big[\begin{smallmatrix} V & E \\
0 & S \end{smallmatrix}\big] \in \gnbh$.
   \end{enumerate}
   \end{lem}
   \begin{proof}
(i) The equality $E=0$ is a direct consequence of
\eqref{gqb-2} and the fact that isometries on
finite-dimensional Hilbert spaces are always unitary.

(ii) This follows from the well known fact that
hyponormal operators in finite-dimensional Hilbert
spaces are normal.
   \end{proof}
   \begin{cor} \label{zobw}
Let $\hh_1$ and $\hh_2$ be separable,
$\big[\begin{smallmatrix} V & E \\
0 & S \end{smallmatrix}\big] \in \gsbh \setminus
\gnbh$, $E\neq 0$ and $N\in\ogr{\kk_2}$ be a minimal
normal extension of $S$. Then
   \begin{align*}
\dim \hh_1 = \dim \hh_2 = \dim \kk_2 =
\dim \kk_2 \ominus \hh_2 = \aleph_0.
   \end{align*}
   \end{cor}
   \begin{proof}
By Lemma~\ref{skwet}, both spaces $\hh_1$ an $\hh_2$
are separable and infinite dimensional. Since
$\big[\begin{smallmatrix} V & E \\
0 & S \end{smallmatrix}\big] \notin \gnbh$, we see
that $\hh_2 \varsubsetneq \kk_2$. It follows from
\cite[Proposition~1($\delta$)]{St-SzIII} that $\dim
\kk_2 = \dim \hh_2$, so $\kk_2$ is separable and
infinite dimensional. Combined with
Corollary~\ref{minfik}(ii), this implies that $\kk_2
\ominus \hh_2$ is separable and infinite dimensional.
This completes the proof.
   \end{proof}
We conclude this section by stating a result on
entrywise extensions of Brown\-ian-type operators of
class $\mathscr{N}$ relative to the first column. It
is closely related to Remark~\ref{cysyk}.
   \begin{lem} \label{nirmun}
Suppose that
$T=\big[\begin{smallmatrix} V & E \\
0 & N \end{smallmatrix}\big] \in \gnbh$, $\dcal$ is a
closed subspace of $\hh_1 \ominus \big(\ob{V} \oplus
\overline{\ob{E}}\big)$ and $\mathfrak m$ is a
cardinal number such that $\mathfrak m\Ge \dim \hh_1$.
Then there exist a Hilbert space $\kk_1\supseteq
\hh_1$ and operators $\tilde V\in \ogr{\kk_1}$ and
$\tilde E\in \ogr{\hh_2,\kk_1}$ such that
$\big[\begin{smallmatrix} \tilde V & \tilde E \\ 0 & N
\end{smallmatrix}\big] \in \gnbhh{\kk_1, \hh_2}$,
$T\preceq \big[\begin{smallmatrix} \tilde V & \tilde E
\\ 0 & N
\end{smallmatrix}\big] $, $|E|=|\tilde E|$, $\dim
\kk_1 = \mathfrak m$ and
   \begin{align*}
\dcal = \kk_1 \ominus (\ob{\tilde V}
\oplus \overline{\ob{\tilde E}}).
   \end{align*}
   \end{lem}
   \begin{proof}
Consider first the case when $\hh_1$ is
finite dimensional. By Lemma~\ref{skwet},
$E=0$, $V$ is unitary and $\dcal =
\{0\}$. Let $\ncal$ be a Hilbert space
such that
   \begin{align*}
\dim \ncal =
   \begin{cases}
\mathfrak m & \text{if } \mathfrak m \Ge
\aleph_0,
   \\
\mathfrak m - \dim \hh_1 &
\text{otherwise.}
   \end{cases}
   \end{align*}
Define the Hilbert space $\kk_1$ and the
operators $\tilde V \in \ogr{\kk_1}$ and
$\tilde E \in \ogr{\hh_2,\kk_1}$ by
   \begin{align} \label{firm}
\left. \begin{aligned} \kk_1&=\hh_1
\oplus \ncal,
   \\
\tilde V (f\oplus g) &= Vf \oplus Wg,
\quad f\in \hh_1, \, g \in \ncal,
   \\
\tilde E f&= E f\oplus 0, \quad f\in
\hh_2,
   \end{aligned}
   \hspace{1ex}\right\}
   \end{align}
where $W\colon \ncal \to \ncal$ is any unitary
operator (cf.\ \eqref{sum}). It is easy to see that
$\big[\begin{smallmatrix} \tilde V & \tilde E \\ 0 & N
\end{smallmatrix}\big]$ meets our requirements.

Suppose now that $\dim \hh_1 \Ge
\aleph_0$. Set $\ecal = \hh_1 \ominus
(\ob{V} \oplus \overline{\ob{E}} \oplus
\dcal)$. Then
   \begin{align}  \label{hikre}
\hh_1=\ob{V} \oplus \overline{\ob{E}}
\oplus \dcal \oplus \ecal.
   \end{align}
Take a Hilbert space $\ncal$ such that $\dim \ncal=\mathfrak
m$. Since $\dim \ecal \Le \dim \hh_1 \Le \mathfrak m$ and
$\dim \hh_1 \Ge \aleph_0$, we have
   \begin{align*}
\mathfrak m=\dim \ncal = \dim \ecal
\oplus \ncal.
   \end{align*}
This implies that there exists a unitary
operator $W \colon \ncal \to \ecal \oplus
\ncal$, where $\ecal \oplus \ncal$ is
regarded as a subspace of $\hh_1 \oplus
\ncal$. Define the Hilbert space $\kk_1$
and the operators $\tilde V \in
\ogr{\kk_1}$ and $\tilde E \in
\ogr{\hh_2,\kk_1}$ by \eqref{firm} with
the current choice of $\ncal$ and $W$. In
view of \eqref{hikre}, $\tilde V$ is well
defined. It is routine to check that in
this case the operator
$\big[\begin{smallmatrix} \tilde V &
\tilde E \\ 0 & N
\end{smallmatrix}\big]$  also meets our requirements.
This completes the proof.
   \end{proof}
   \section{\label{Sec2}Proofs of Theorems~\ref{etrest}, \ref{mudyl}
and \ref{mude}} We begin by providing a basic construction of
taut entrywise extensions of operators of class
$\mathscr{S}$.
   \begin{bc*}
Suppose that $T=\big[\begin{smallmatrix} V & E \\
0 & S \end{smallmatrix}\big] \in \gsbh$ and $N\in
\ogr{\kk_2}$ is a minimal normal extension of $S$. We will
give a method of constructing entrywise extensions
$\big[\begin{smallmatrix} \tilde V & \tilde E \\ 0 & N
\end{smallmatrix}\big] \in \gnbhh{\kk_1, \kk_2}$ of $T$ such
that $|E| \subseteq |\tilde E|$.

According to the Bram lifting theorem (see
\cite[Theorem~II.10.5]{Co91}), there exists a map $\psi$ such
that \allowdisplaybreaks
   \begin{align} \label{pusyya}
&\textit{$\psi\colon \{S,S^*\}' \to \{N\}' \cap \{P\}'$ is a
$*$-isomorphism and}
   \\ \label{pusyy}
&\textit{$A \subseteq \psi(A)$ for all $A \in \{S,S^*\}'$,}
   \end{align}
where $P\in \ogr{\kk_2}$ is the orthogonal projection of
$\kk_2$ onto $\hh_2$ (recall that
$\psi(I_{\hh_2})=I_{\kk_2}$). Observe that if $A \in
\{S,S^*\}'$ and $A\Ge 0$, then by the square root theorem
$A^{1/2} \in \{S,S^*\}'$. Hence, by \eqref{pusyya}, $\psi(A)
= \psi(A^{1/2})^*\psi(A^{1/2}) \Ge 0$, which shows that
   \begin{align} \label{pusyt}
\textit{if $A \in \{S,S^*\}'$ and $A\Ge 0$, then $\psi(A) \Ge
0$.}
   \end{align}
Next, note that by \eqref{gqb-3b}, $|E| \in \{S,S^*\}'$.
Thus, by \eqref{pusyya}-\eqref{pusyt} we have
\allowdisplaybreaks
   \begin{gather} \label{dcsweq}
|E| \subseteq \psi(|E|),
   \\ \label{dcsweq15}
\psi(|E|) \Ge 0,
   \\ \label{dcsweq3}
|E|^2 \subseteq \psi(|E|^2) =
\psi(|E|)^*\psi(|E|)=\psi(|E|)^2,
   \\  \label{dcsweq13}
\psi(|E|)N = N \psi(|E|).
   \end{gather}
Let $\mcal_0$ be a Hilbert space of the same dimension as
$\kk_2 \ominus \overline{\ob{|E|}}$. Take a unitary operator
$U_2\colon \kk_2 \ominus \overline{\ob{|E|}} \to \mcal_0$.
Set $\ecal = \hh_1 \ominus (\ob{V} \oplus \overline{\ob{E}}
\oplus \dcal)$. Then
   \begin{align} \label{dipost}
\hh_1=\ob{V} \oplus \overline{\ob{E}} \oplus \dcal \oplus
\ecal.
   \end{align}
Observe that there exists a Hilbert space $\mcal_1$ such that
   \begin{align} \label{cimosz}
\dim \ecal \oplus \mcal_1 = \dim \mcal_0 \oplus \mcal_1.
   \end{align}
Take an arbitrary unitary operator
   \begin{align*}
W\colon \mcal_0 \oplus \mcal_1 \to \ecal \oplus \mcal_1,
   \end{align*}
where $\ecal \oplus \mcal_1$ is regarded as a subspace of
$\hh_1 \oplus \mcal_1$. Set
   \begin{align} \label{krk1}
\kk_1= \hh_1 \oplus \mcal_0 \oplus \mcal_1.
   \end{align}
Let $E=U_1|E|$ be the polar decomposition of $E$. Since
$U_1|_{\overline{\ob{|E|}}}$ is an isometry and
$U_1(\overline{\ob{|E|}}) = \overline{\ob{E}}$, we see that
the map $U\colon \kk_2 \to \kk_1$ given by
   \begin{align} \label{shhands}
U(f \oplus g) = U_1 f\oplus U_2 g, \quad f \in
\overline{\ob{|E|}}, \, g \in \kk_2 \ominus
\overline{\ob{|E|}},
   \end{align}
is a well defined isometric operator. Define the operator
$\tilde E \in \ogr{\kk_2,\kk_1}$ by
   \begin{align} \label{dfesq}
\tilde E=U \psi(|E|).
   \end{align}
Note that
   \begin{align*}
\tilde E f \overset{\eqref{dcsweq}}= U |E| f
\overset{\eqref{shhands}}= U_1 |E| f = Ef, \quad f\in \hh_2.
   \end{align*}
This shows that $E \subseteq \tilde E$. Since by
\eqref{dcsweq15}, $\psi(|E|)^*=\psi(|E|)$ we deduce that
   \begin{align} \label{dcsweq11}
|E|^2 \overset{\eqref{dcsweq3}}\subseteq \psi(|E|)^*
\psi(|E|) \overset{\eqref{dfesq}}=\tilde E^* \tilde E =
|\tilde E|^2,
   \end{align}
which by Lemma~\ref{inrx} yields \eqref{pruypec}. We infer
from \eqref{dcsweq3}, \eqref{dcsweq13} and \eqref{dcsweq11}
that $N$ commutes with $\tilde E^* \tilde E$. Define the
isometric operator $\tilde V \in \ogr{\kk_1}$ by
   \begin{align} \label{nutusse}
\tilde V (f \oplus g \oplus h) = Vf \oplus W(g\oplus h),
\quad f \in \hh_1, \, g \in \mcal_0, \, h \in \mcal_1.
   \end{align}
Clearly $V \subseteq \tilde V$. Since
$U_1(\overline{\ob{|E|}}) = \overline{\ob{E}}$, we infer from
\eqref{shhands} that
   \begin{align} \label{drubc}
\ob{U} = \overline{\ob{E}} \oplus \mcal_0 \oplus \{0\}.
   \end{align}
On the other hand, by \eqref{nutusse}, we have
   \begin{align} \label{drubcer}
\ob{\tilde{V}} = (\ob{V} \oplus \ecal) \oplus \{0\} \oplus
\mcal_1.
   \end{align}
Hence, by \eqref{gqb-2} and \eqref{dfesq}, the ranges of
$\tilde V$ and $\tilde E$ are orthogonal. As a consequence,
$\tilde V^* \tilde E = 0$. This shows that
$\big[\begin{smallmatrix} \tilde V & \tilde E \\ 0 & N
\end{smallmatrix}\big]$ has the
desired properties.
      \hfill $\diamondsuit$
   \end{bc*}
   Now we are ready to prove the main theorem.
   \begin{proof}[Proof of Theorem~\ref{etrest}]
(ii)$\Rightarrow$(i) Since $\big[\begin{smallmatrix} \tilde V
& \tilde E \\ 0 & N \end{smallmatrix}\big] \in \gnbhh{\kk_1,
\kk_2}$ is an entrywise extension
of $\big[\begin{smallmatrix} V & E \\
0 & S \end{smallmatrix}\big]$, it follows
that $V$ is an isometry, $S$ is subnormal
with the normal extension $N$ and the
ranges of $V$ and $E$ are orthogonal.
This implies that the conditions
\eqref{gqb-1} and \eqref{gqb-2} are
satisfied. The condition \eqref{gqb-3b}
(with $X=S$) is a consequence of
\eqref{pruypec} and the assumption that
$N$ commutes with $\tilde E^* \tilde E$.
This shows that (i) holds.

Next we prove (a) and (b). For, assume that
$\big[\begin{smallmatrix} \tilde V & \tilde E \\
0 & N \end{smallmatrix}\big]$ is as in (ii) and
$N=\mathrm{mne}\,S$.

   (a) Since by the square root theorem $N$ and
$|\tilde E|$ commute, we infer from \eqref{pruypec}
that $|\tilde E| \in \{N\}' \cap \{P\}'$. This implies
that $|\tilde E|$ is the lift of $|E|$ relative to
$\{N\}'$, i.e.,
   \begin{align} \label{lufy}
|\tilde E| = \psi(|E|),
   \end{align}
where $\psi$ is as in \eqref{pusyya} and \eqref{pusyy}.

   (b) The case of injectivity is a
consequence of \eqref{luinw},
\eqref{pusyya}, \eqref{lufy} and
Lemma~\ref{inyec} applied to von Neumann
algebras $\acal:=\{S,S^*\}^{\prime}$ and
$\bcal:=\{N\}^{\prime} \cap
\{P\}^{\prime}$ and the positive element
$A:=|E|$. In turn, the left-invertibility
case can be deduced from \eqref{luinw},
\eqref{pusyya}, \eqref{lufy} and
\cite[Theorem~11.29]{Rud73}. Finally, the
case of isometricity goes as follows: if
$E$ is an isometry, then by (a), $|\tilde
E|$ is the lift of $|E|=I_{\hh_2}$
relative to $\{N\}'$, so by the
uniqueness of lifts, $|\tilde
E|=I_{\kk_2}$, which means that $\tilde
E$ is an isometry (the converse is
trivial because $E \subseteq \tilde E$).

(i)$\Rightarrow$(iii) We begin by showing
that (iii) holds except for the equation
   \begin{align} \label{dumriq}
\dim \kk_1 = \dim \hh_1.
   \end{align}
Note that the other required equality
$\dim\kk_2 = \dim \hh_2$ follows from the
general fact that if $N\in \ogr{\kk}$ is
a minimal normal extension of a subnormal
operator $S\in \ogr{\hh}$, then $\dim
\kk=\dim \hh$ (see e.g.,
\cite[Proposition~1($\delta$)]{St-SzIII}).

Let us first consider the case when $E$
is injective. Suppose that $\big[\begin{smallmatrix} V & E \\
0 & S \end{smallmatrix}\big] \in \gsbh$ and $\jd{E}=\{0\}$.
Let $N\in \ogr{\kk_2}$ be a minimal normal extension of $S$
and let $\kk_1$, $U$, $\tilde V$ and $\tilde E$ be as in
Basic Construction. According to \eqref{luinw} and (b), the
operator $|\tilde E|$ is injective. Hence, by \eqref{lufy},
$\overline{\ob{\psi(|E|)}} = \kk_2$. Consequently, in view of
\eqref{dfesq}, $\ob{U}=\overline{\ob{\tilde E}}$. Combined
with \eqref{drubc} and \eqref{drubcer}, this implies~that
   \begin{align*}
\ob{\tilde V} \oplus \overline{\ob{\tilde
E}} = (\ob{V} \oplus \overline{\ob{E}}
\oplus \ecal) \oplus \mcal_0 \oplus
\mcal_1.
   \end{align*}
Hence, by \eqref{dipost}, the equality
\eqref{rozkly} holds.

Consider now the case when $E$ is not injective. By
\eqref{reotp}, $\jd{E}=\jd{|E|}$ reduces $S$ and $|E|$. Let
$S=S_1 \oplus S_2$ and $E=E_1 \oplus E_2$ be the orthogonal
decompositions of $S$ and $E$ relative to the orthogonal
decomposition $\hh_2 = \jd{E} \oplus \overline{\ob{|E|}}$,
where $E_1\colon \jd{E} \to \{0\}$ and $E_2\colon
\overline{\ob{|E|}} \to \hh_1$. Clearly, both operators $S_1$
and $S_2$ are subnormal. It is easy to see that
$\big[\begin{smallmatrix} V & E_2 \\
0 & S_2 \end{smallmatrix}\big] \in
\mathscr{S}_{\hh_1,\overline{\ob{|E|}}}$ and $\jd{E_2} =
\{0\}$. Since $\ob{E}=\ob{E_2}$, we obtain
   \begin{align} \label{ztugi}
\dcal \subseteq \hh_1 \ominus \big(\ob{V}
\oplus \overline{\ob{E}}\big) = \hh_1
\ominus \big(\ob{V} \oplus
\overline{\ob{E_2}}\big).
   \end{align}
Let $N\in \ogr{\kk_2}$ be a minimal
normal extension of $S$. By
Lemma~\ref{simurt}, $N$ has the
orthogonal decomposition $N=N_1 \oplus
N_2$ relative to an orthogonal
decomposition $\kk_2=\kk_{21} \oplus
\kk_{22}$, where $N_j=\mathrm{mne}\, S_j$
for $j=1,2$. Using \eqref{ztugi} and
applying the case
of injectivity to $\big[\begin{smallmatrix} V & E_2 \\
0 & S_2 \end{smallmatrix}\big]$ and
$N_2$, we deduce that there exist a
Hilbert space $\kk_1$ and operators
$\tilde V\in \ogr{\kk_1}$ and $\tilde E_2
\in \ogr{\kk_{22},\kk_1}$ such that
$\big[\begin{smallmatrix} \tilde V &
\tilde E_2 \\ 0 & N_2
\end{smallmatrix}\big] \in \gnbhh{\kk_1,
\kk_{22}}$,  $\big[\begin{smallmatrix} V & E_2 \\
0 & S_2 \end{smallmatrix}\big] \preceq
\big[\begin{smallmatrix} \tilde V &
\tilde E_2 \\ 0 & N_2
\end{smallmatrix}\big]$,
$|E_2| \subseteq |\tilde E_2|$ and
   \begin{align} \label{slonce}
\dcal=\kk_1 \ominus (\ob{\tilde V} \oplus
\overline{\ob{\tilde E_2}}).
   \end{align}
Define $\tilde E \in \ogr{\kk_2,\kk_1}$
by $\tilde E = 0 \oplus \tilde E_2$,
where now $0$ is the zero operator from
$\kk_{21}$ to $\{0\}$. It is a matter of
routine to verify that
$\big[\begin{smallmatrix} \tilde V &
\tilde E \\ 0 & N
\end{smallmatrix}\big] \in \gnbhh{\kk_1,
\kk_2}$ and that
$\big[\begin{smallmatrix} \tilde V &
\tilde E \\ 0 & N
\end{smallmatrix}\big]$ is
an entrywise extension of $\big[\begin{smallmatrix} V & E \\
0 & S \end{smallmatrix}\big]$ that
satisfies \eqref{pruypec}. Since
$\ob{\tilde E} = \ob{\tilde E_2}$, we see
that \eqref{rozkly} follows from
\eqref{slonce}.

It remains to show that the space $\kk_1$ can be chosen so
that \eqref{dumriq} holds. Since the case when $S$ is normal
follows from Lemma~\ref{nirmun} applied to $\mathfrak m=\dim
\hh_1$, we can assume that $S$ is not normal. In view of what
was proved in the previous paragraph, it is enough to
consider the case when $E$ is injective. Let $\kk_1$, $U$,
$\tilde V$ and $\tilde E$ be as in Basic Construction.
Observe that $\dim \kk_2 \ominus \overline{\ob{|E|}} \Ge
\aleph_0$ (otherwise $\dim \kk_2 \ominus \hh_2 < \aleph_0$,
so by Corollary~\ref{minfik}(ii), $S$ is normal, a
contradiction). Recall that $\dim \kk_2 = \dim \hh_2$. Since
   \begin{align} \label{sar1}
\aleph_0 \Le \dim \mcal_0 = \dim \kk_2
\ominus \overline{\ob{|E|}} \Le \dim
\kk_2 = \dim \hh_2,
   \end{align}
and by \eqref{dipost}, $\dim \ecal \Le
\dim \hh_1$, the space $\mcal_1$
satisfying \eqref{cimosz} can be chosen
so that
   \begin{align*}
\dim \mcal_1 = \max\{\dim \hh_1,
\dim\hh_2\}.
   \end{align*}
Combined with \eqref{krk1} and
\eqref{sar1}, this implies that
   \begin{align*}
\dim \kk_1 = \max\{\dim \hh_1,
\dim\hh_2\}.
   \end{align*}
Since by \eqref{luinw}, $|E|$ is
injective and consequently
$\overline{\ob{|E|}}=\hh_2$, we infer
from \eqref{tres} that \eqref{dumriq} is
valid.

(iii)$\Rightarrow$(ii) Obvious.

This completes the proof.
   \end{proof}
   \begin{rem}
Regarding the proof of Theorem~\ref{etrest}, we observe the
following.
   \begin{align} \label{ebifol}
   \begin{minipage}{75ex}
{\em If $A \in \{S,S^*\}'$, $A \Ge 0$ and $\alpha
>0$, then $A^{\alpha} \in \{S,S^*\}'$ and
$\psi(A^{\alpha})=\psi(A)^{\alpha}$}.
   \end{minipage}
   \end{align}
Consequently, $\psi(|X|)=|\psi(X)|$ for every $X \in
\{S,S^*\}'$. These are general facts which are valid for
$*$-homomorphisms of $C^*$-algebras (cf.\ the proof of
Lemma~\ref{inrx}). In particular, we have
   \begin{align*}
\psi(|E|^{\alpha})=\psi(|E|)^{\alpha}
\overset{\eqref{lufy}}=|\tilde E|^{\alpha}.
   \end{align*}
The assertion \eqref{ebifol} can also be deduced from
\eqref{pusyya}, \eqref{pusyy}, Lemma~\ref{inrx} and the
uniqueness of lifts.

It is also worth pointing out that in view of \eqref{dfesq}
and \eqref{lufy},
   \begin{align*}
\tilde E= U \psi(|E|) = U |\tilde E|.
   \end{align*}
This resembles the polar decomposition of $\tilde E$.
However, since $U$ is an isometry, it can be verified that
$U$ coincides with the partial isometry factor in the polar
decomposition of $\tilde E$ if and only if $\tilde E$ is
injective. \hfill $\diamondsuit$
   \end{rem}
   \begin{proof}[Proof of Theorem~\ref{mudyl}]
In view of \eqref{luinw}, $|E|$ is injective and thus
$\overline{\ob{|E|}} = \hh_2$, which implies that $W$
is an isometry.

(i) Let $\big[\begin{smallmatrix} \tilde V & \tilde E
\\ 0 & N \end{smallmatrix}\big] \in \gnbhh{\kk_1,
\kk_2}$ be an entrywise extension of $T$ such that
$|E| \subseteq |\tilde E|$. Then, by
Theorem~\ref{etrest}, $|\tilde E|$ is the lift of
$|E|$ relative to $\{N\}'$ and $|\tilde E|$ is
injective. This implies that $\overline{\ob{|\tilde
E|}}=\kk_2$. As a consequence, if $\tilde E=U|\tilde
E|$ is the polar decomposition of $\tilde E$, then
$\tilde E$ is injective, $U\in \ogr{\kk_2,\kk_1}$ is
an isometry, $\ob{U}=\overline{\ob{\tilde E}}$ and
$\ob{\tilde V} \perp \ob{U}$. Since
$\overline{\ob{|E|}} = \hh_2$ and
   \begin{align*}
W(|E|f) = Ef = \tilde E f = U|\tilde{E}|f
\overset{(*)}= U(|E|f), \quad f \in \hh_2,
   \end{align*}
where $(*)$ follows from $|E|\subseteq |\tilde{E}|$,
we conclude that $W \subseteq U$.

(ii) It follows from \eqref{pusyt} that the operator
$B$ is positive. Set $\tilde E=UB$. Then, by the
square root theorem, $|\tilde E|=B$. Since $|E|$ is
injective, we infer from Lemma~\ref{inyec} and
\eqref{pusyya} that $B$ is injective as well. As a
consequence, $\overline{\ob{B}}=\kk_2$, which implies
that $\tilde E=UB$ is the polar decomposition of
$\tilde{E}$. In particular,
$\overline{\ob{\tilde{E}}}=\ob{U}$.
Therefore, we see that $\big[\begin{smallmatrix} \tilde V & \tilde E \\
0 & N \end{smallmatrix}\big] \in \gnbhh{\kk_1,
\kk_2}$, $\big[\begin{smallmatrix} V & E \\ 0 & S
\end{smallmatrix}\big] \preceq \big[\begin{smallmatrix}
\tilde V & \tilde E \\ 0 & N \end{smallmatrix}\big]$
and $|E| \subseteq |\tilde E|$. This completes the
proof.
   \end{proof}
   \begin{proof}[Proof of Theorem~\ref{mude}]
The conditions (a) and (b) follow from \eqref{reotp}. The
condition (c) is a direct consequence of (b). Fix $j\in
\{1,2\}$. It follows from (c) that $\hh_{2j}$ reduces $|E|$
and $|E| = |E_1| \oplus |E_2|$, so by $|E| \in \{S\}'$, we
see that $|E_j| \in \{S_j\}'$. Since $\ob{E}=\ob{E_2}$, we
deduce that $\big[\begin{smallmatrix} V & E_j \\
0 & S_j \end{smallmatrix}\big] \in
\gsbhh{\hh_1, \hh_{2j}}$. By
Theorem~\ref{etrest}(a), $|\tilde E|$ is
the lift of $|E|$ relative to $\{N\}'$.
It follows from Lemma~\ref{simurt} that
$\kk_{2j}$ reduces $|\tilde E|$ and
$|\tilde E||_{\kk_{2j}}$ is the lift of
$|E_j|$ relative to $\{N_j\}'$. Hence,
we~have
   \begin{align*}
\is{|\tilde E_j|^2 f}{g}=\is{\tilde E_j
f}{\tilde E_jg} = \is{\tilde E f}{\tilde
Eg} = \is{(|\tilde
E||_{\kk_{2j}})^2f}{g}, \quad f,g \in
\kk_{2j}.
   \end{align*}
so by the uniqueness of square roots,
$|\tilde E_j|=|\tilde E||_{\kk_{2j}}$.
This implies that $|\tilde E_j|$ is the
lift of $|E_j|$ relative to $\{N_j\}'$
and, in view of $\ob{\tilde E_j}\subseteq
\ob{\tilde E}$, that
$\big[\begin{smallmatrix} \tilde V &
\tilde E_j
\\ 0 & N_j \end{smallmatrix}\big] \in
\gnbhh{\kk_1, \kk_{2j}}$. This proves
(d), (e) and (f). Since $|E_1|=0$ and
$|\tilde E_1|$ is the lift of $|E_1|$
relative to $\{N_1\}'$, we deduce that
$|\tilde E_1|=0$. In turn, because $E_2$
is injective, we infer from
Theorem~\ref{etrest}(b) that $\tilde E_2$
is injective. Hence (g) holds. This
completes the proof.
   \end{proof}
   \section{\label{Sec.4}Examples}
In this section, we give four examples illustrating
the subject of the paper. The first shows that
operators of class $\mathscr{S}$ may not be subnormal,
the second deals with extensions that are not
entrywise, the third discusses the relationship
between the defect spaces of an operator of class
$\mathscr{S}$ and its taut entrywise extension, and
finally the fourth shows that powers of operators of
class $\mathscr{S}$ may not be of class $\mathscr{S}$.

We begin by providing examples of operators of class
$\mathscr{S}$ (or even of class $\mathscr{U}$), which
do not have extensions to normal operators, but,
according to Theorem~\ref{etrest}, have taut entrywise
extensions to operators of class $\mathscr{N}$ (cf.\
Lemma \ref{nonur}).
   \begin{exa}\label{noistx}
Suppose that $V\in \ogr{\hh_1}$, $U\in \ogr{\hh_2,\hh_1}$ and
$Q\in \ogr{\hh_2}$ are isometries such that $\ob{V}\perp
\ob{U}$. Let $\alpha$ be a nonzero complex number.
Then $T_{\alpha}=\big[\begin{smallmatrix} V & \alpha U \\
0 & Q \end{smallmatrix}\big] \in\gibh$. Since $\|Q\|^2 +
|\alpha|^2 > 1$, it follows from
\cite[Corollary~5.3]{C-J-J-S21} that $T_{\alpha}$ is not
subnormal. In particular, if $Q$ is unitary, then
$T_{\alpha}$ is a Brownian isometry, and if $Q$ is unitary,
$|\alpha|=1$ and $\hh_1=\ob{V}\oplus \ob{U}$, then
$T_{\alpha}$ is a Brownian unitary (see the next paragraph
for definition).
   \hfill $\diamondsuit$
   \end{exa}
It is well-known that any $2$-isometry has an extension to a
Brownian unitary (see \cite[Theorem~5.80]{Ag-St}). However,
as shown in Example~\ref{noistn} below, there are
$2$-isometric Brownian-type operators of class $\mathscr{Q}$
which have no entrywise extensions to Brownian unitaries.
Hence, in general, extensions may not be entrywise. Before
providing such an example, we recall that the quantity
$\sqrt{\|A^*A-I\|}$ is called the {\em covariance} of an
operator $A\in\ogr{\hh}$. Given $\sigma \in (0,\infty)$, we
say that an operator $T$ is a {\em Brownian unitary of
covariance} $\sigma$ (see \cite[Proposition~5.12]{Ag-St}) if
$T=\big[\begin{smallmatrix} X & \sigma Y \\
0 & Z \end{smallmatrix}\big] \in \ogr{\hh_1 \oplus
\hh_2}$, where
   \begin{align}\label{cyrul}
   \begin{minipage}{74ex}
$X\in \ogr{\hh_1}$ and $Y\in\ogr{\hh_2,\hh_1}$ are
isometries such that $\ob{X}=\jd{Y^*}$ and $Z\in
\ogr{\hh_2}$ is unitary.
   \end{minipage}
   \end{align}
By a {\em Brownian unitary of covariance} $0$, we mean any
unitary operator. In both cases $\sigma^2=\|T^*T-I\|$.
Clearly, each Brownian unitary of covariance $\sigma > 0$ is
of class $\mathscr{U}$.
   \begin{exa} \label{noistn}
Suppose that $\lcal$, $\mcal$ and $\ncal$ are separable
infinite dimensional Hilbert spaces. Let $U_1, U_2\in
\ogr{\mcal}$ be isometries such that $U_1$ is not unitary,
$D\in \ogr{\lcal}$ be a nonzero operator and $A\in
\ogr{\ncal}$ be a selfadjoint operator such that $0 \Le A \Le
I$. Then $I-A^2$ is a positive contraction. Define the
Hilbert spaces $\hh_1$ and $\hh_2$, and the operators $E \in
\ogr{\hh_2,\hh_1}$ and $Q\in \ogr{\hh_2}$ by
   \allowdisplaybreaks
   \begin{align} \notag
\hh_i&=\lcal \oplus (\mcal \otimes \ncal), \quad
i=1,2,
   \\ \label{samsam}
E &= D \oplus (U_1\otimes \sqrt{I-A^2}),
   \\ \label{qqqqq}
Q&=I_{\lcal} \oplus (U_2\otimes A).
   \end{align}
Note that if $\ff$ and $\kk$ are Hilbert spaces, and
$\ff_1$ and $\kk_1$ are their closed subspaces,
respectively, then
   \begin{align} \label{terten}
(\ff \otimes \kk) \ominus (\ff_1 \otimes \kk_1) =
[(\ff\ominus \ff_1) \otimes \kk] \oplus [\ff_1 \otimes
(\kk \ominus \kk_1)].
   \end{align}
Since $I+A$ is injective and consequently
$\jd{I-A^2}=\jd{I-A}$, it follows from \eqref{terten}
applied to $\ff=\mcal$, $\kk=\ncal$, $\ff_1=\ob{U_1}$
and $\kk_1=\overline{\ob{\sqrt{I-A^2}}}$ that
   \allowdisplaybreaks
   \begin{align*}
\hh_1 \ominus \overline{\ob{E}} & = [\lcal \ominus
\overline{\ob{D}}] \oplus [\jd{U_1^*} \otimes \ncal]
\oplus [\ob{U_1} \otimes \jd{\sqrt{I-A^2}}]
   \\
& = [\lcal \ominus \overline{\ob{D}}] \oplus
[\jd{U_1^*} \otimes \ncal] \oplus [\ob{U_1} \otimes
\jd{I-A}].
   \end{align*}
Since $U_1$ is not unitary, $\dim \jd{U_1^*} \Ge 1$, so $\dim
(\jd{U_1^*} \otimes \ncal)=\aleph_0$. Combined with $\dim
\hh_1=\aleph_0$, this implies that there exists a non-unitary
isometry $V\in \ogr{\hh_1}$ such that $\ob{V} \perp \ob{E}$.
It is a matter of routine to verify that  $T:= \big[\begin{smallmatrix} V & E \\
0 & Q \end{smallmatrix}\big] \in \gqbh$~and
   \begin{align*}
(|Q|^2-I)(|Q|^2 + |E|^2 - I)=0.
   \end{align*}
Using \cite[Theorem~7.1]{C-J-J-S21}, we conclude that
$T$ is a $2$-isometry of covariance $\|D\|>0$ (in
particular, $T$ is not an isometry). It follows from
\cite[Theorem~5.80]{Ag-St} that $T$ has an extension
to a Brownian unitary of covariance $\|D\|$. We show
that $T$ has no entrywise extension to a Brownian
unitary of arbitrary covariance $\sigma$. Since $T$ is
not an isometry, it is sufficient to consider the case
when $\sigma >0$. Suppose, on the contrary, that $T$
has an entrywise extension to a Brownian unitary of
covariance $\sigma > 0$, that is, $T\preceq
\big[\begin{smallmatrix} X & \sigma Y \\ 0 & Z
\end{smallmatrix}\big]$ with $X$, $Y$ and $Z$
satisfying \eqref{cyrul}. In particular, we have
$Q\subseteq Z$ and $E\subseteq \sigma Y$. Since $Z$ is
unitary, $Q$ must be an isometry. Therefore, by
\eqref{qqqqq}, $U_2 \otimes A$ is an isometry, which
implies that $A=I_{\ncal}$. As a consequence of
\eqref{samsam}, $E= D \oplus 0$, where $0$ is the zero
operator on the infinite dimensional Hilbert space
$\mcal \otimes \ncal$. However, this contradicts
$E\subseteq \sigma Y$ (because, by \eqref{cyrul}, $Y$
is an isometry).~\hfill $\diamondsuit$
   \end{exa}
In Example~\ref{tryisu} below, we show that there
exist Brownian-type operators $T=\big[\begin{smallmatrix} V & E \\
0 & S \end{smallmatrix}\big] \in \gsbh$ having taut
entrywise extensions
$\big[\begin{smallmatrix} \tilde V & \tilde E \\
0 & N \end{smallmatrix}\big] \in \gnbhh{\kk_1, \kk_2}$
for which the dimensions of the defect spaces $\hh_1
\ominus \big(\ob{V} \oplus \overline{\ob{E}}\big)$ and
$\kk_1 \ominus (\ob{\tilde V} \oplus
\overline{\ob{\tilde E}})$ can be arbitrary.
   \begin{exa} \label{tryisu}
For the sake of simplicity, we will assume that $\hh_1$ and
$\hh_2$ are separable infinite dimensional Hilbert spaces
(cf.\ Corollary~\ref{zobw}), that is,
   \begin{align} \label{sumpi}
\dim \hh_1 = \dim \hh_2 = \aleph_0.
   \end{align}
Under this assumption, for a given $\mathfrak p\in
\{0,1,2, \ldots\} \cup \{\aleph_0\}$, it is possible
to construct all Brownian-type operators
$\big[\begin{smallmatrix} V & E \\ 0 & S
\end{smallmatrix}\big] \in \gsbh$ with isometric
$E$ such that
   \begin{align} \label{dumi}
\dim \hh_1 \ominus \big(\ob{V} \oplus \ob{E}\big) =
\mathfrak p.
   \end{align}
Namely, take a subnormal operator $S\in \ogr{\hh_2}$
and two isometries $V\in \ogr{\hh_1}$ and $E\in
\ogr{\hh_2,\hh_1}$ such that $\ob{V} \perp \ob{E}$ and
\eqref{dumi} holds (this is possible due to
\eqref{sumpi}). Then clearly
$\big[\begin{smallmatrix} V & E \\
0 & S \end{smallmatrix}\big] \in \gsbh$. By \eqref{gqb-1} ad
\eqref{gqb-2}, all operators of class $\mathscr{S}$ with
isometric $E$ can be obtained in this way.

Let $\mathfrak m, \mathfrak n \in \{0,1,2, \ldots\}
\cup \{\aleph_0\}$. In view of the above discussion
and \eqref{sumpi},
there is $T=\big[\begin{smallmatrix} V & E \\
0 & S \end{smallmatrix}\big] \in \gsbh$ with isometric $E$,
$S$ is not normal and
   \begin{align} \label{dindsp}
\dim \hh_1 \ominus \big(\ob{V} \oplus \ob{E}\big) =
\mathfrak m.
   \end{align}
We will show that there exists an entrywise extension
$\big[\begin{smallmatrix} \tilde V & \tilde E \\ 0 & N
\end{smallmatrix}\big] \in \gnbhh{\kk_1, \kk_2}$  of
$T$ with isometric $\tilde E$ such that
$N=\mathrm{mne}\, S$ for which
   \begin{align} \label{utry}
\dim \kk_1 \ominus (\ob{\tilde V} \oplus \ob{\tilde
E}) = \mathfrak n.
   \end{align}

For, let $N\in \ogr{\kk_2}$ be a minimal normal
extension of $S$. Consider two cases.

{\sc Case 1.} $\mathfrak{n} \Le \mathfrak{m}$.

By Theorem~\ref{etrest} applied to the space $\dcal$ of
dimension $\mathfrak{n}$, $T$ has an entrywise extension
$\big[\begin{smallmatrix} \tilde V & \tilde E \\ 0 & N
\end{smallmatrix}\big] \in \gnbhh{\kk_1, \kk_2}$ such
that $\tilde E$ is an isometry (in particular, $|E|
\subseteq |\tilde E|$), the equality \eqref{utry} is
valid and
   \begin{gather} \label{nimulib}
\dim \kk_1 = \dim \hh_1 \overset{\eqref{sumpi}}=
\aleph_0.
   \end{gather}

{\sc Case 2.} $\mathfrak{n} > \mathfrak{m}$.

Set $\dcal=\hh_1 \ominus \big(\ob{V} \oplus \ob{E}\big)$.
Applying Theorem~\ref{etrest} again, we see that $T$ has an
entrywise extension $\big[\begin{smallmatrix} \tilde V &
\tilde E \\ 0 & N
\end{smallmatrix}\big] \in \gnbhh{\kk_1, \kk_2}$ such
that $\tilde E$ is an isometry (in particular, $|E|
\subseteq |\tilde E|$), the equality \eqref{nimulib}
is valid and
   \begin{align*}
\dcal=\kk_1 \ominus (\ob{\tilde V} \oplus \ob{\tilde
E}).
   \end{align*}
Next, it follows from Remark~\ref{cysyk} applied to
$\lcal=\ell^2$ that there exist a Hilbert space
$\mcal$ and $\big[\begin{smallmatrix} \tilde U &
\tilde F \\ 0 & N \end{smallmatrix}\big] \in
\gnbhh{\mcal, \kk_2}$ such that $T \preceq
\big[\begin{smallmatrix} \tilde V & \tilde E \\ 0 & N
\end{smallmatrix}\big] \preceq \big[\begin{smallmatrix}
\tilde U & \tilde F \\ 0 & N
\end{smallmatrix}\big]$,  $|\tilde F|=|\tilde E|=I$,
$\dim \mcal=\aleph_0$~and
   \begin{align*}
\mcal \ominus (\ob{\tilde U} \oplus \ob{\tilde F}) =
\dcal \oplus \ncal,
   \end{align*}
where $\ncal$ is a Hilbert space such that
$\dim{\ncal}=\mathfrak{n}-\mathfrak{m}$. Consequently,
$\big[\begin{smallmatrix} \tilde U & \tilde F
\\ 0 & N \end{smallmatrix}\big] \in \gnbhh{\mcal,
\kk_2}$ is a taut entrywise extension of $T$ for which
   \begin{align*}
\dim \mcal \ominus (\ob{\tilde U} \oplus \ob{\tilde
F}) \overset{\eqref{dindsp}}=\mathfrak{n}.
   \end{align*}

Finally, note that in both cases the following
equalities are valid:
   \begin{align*}
\dim \kk_2 = \dim \kk_2 \ominus \hh_2 = \aleph_0 =
\dim \kk_1.
   \end{align*}
This is an immediate consequence of
Corollary~\ref{zobw} and \eqref{nimulib}.
   \hfill $\diamondsuit$
   \end{exa}
By a {\em unilateral weighted shift with weights}
$\{\lambda_i\}_{i=0}^{\infty}$, we mean a uniquely
determined operator $W\in \ogr{\ell^2}$ such that
   \begin{align*}
We_i=\lambda_i e_{i+1}, \quad i\Ge 0,
   \end{align*}
where $\{\lambda_i\}_{i=0}^{\infty}$ is a sequence of complex
numbers and $\{e_i\}_{i=0}^{\infty}$ denotes the standard
orthonormal basis of $\ell^2$. Before stating the fourth
example, we need the following characterization of the
condition (ii) of Proposition~\ref{xyzzyv} in the case of
unilateral weighted shifts.
   \begin{pro} \label{suwt}
If $W\in \ogr{\ell^2}$ is an injective hyponormal
unilateral weighted shift with weights
$\{\lambda_i\}_{i=0}^{\infty}$ and $n\Ge 2$, then the
following statements are equivalent{\em :}
   \begin{enumerate}
   \item[(i)] the operators $W^n$ and
$\sum_{j=1}^{n-1}W^{*j}W^j$ commute,
   \item[(ii)] the sequence
$\{|\lambda_i|\}_{i=0}^{\infty}$ is constant.
   \end{enumerate}
   \end{pro}
   \begin{proof}
Since $W^{*j}W^j e_k = \Big(\prod_{l=k}^{k+j-1}
|\lambda_l|^2\Big)e_k$ for all $k\Ge 0$ and $j\Ge 1$,
and $\lambda_i\neq 0$ for all $i\Ge 0$, we see that
$W^n$ and $\sum_{j=1}^{n-1}W^{*j}W^j$ commute if and
only if
   \begin{align} \label{hrypa}
\gamma_k:=\sum_{j=1}^{n-1} \Big(\prod_{l=k}^{k+j-1}
|\lambda_l|^2\Big) = \sum_{j=1}^{n-1}
\Big(\prod_{l=k+n}^{k+n+j-1} |\lambda_l|^2\Big), \quad
k \Ge 0.
   \end{align}
Suppose that \eqref{hrypa} holds. Since $W$ is
hyponormal, we see that
   \begin{align} \label{hdyw}
|\lambda_l| \Le |\lambda_{l+1}|, \quad l\Ge 0.
   \end{align}
Hence, we have
   \begin{align*}
\gamma_k \Le \gamma_{k+1} \Le \ldots \Le \gamma_{k+n}
\overset{\eqref{hrypa}}= \gamma_k, \quad k\Ge 0.
   \end{align*}
Thus
   \begin{align*}
0 = \gamma_{k+1} - \gamma_k = \sum_{j=1}^{n-1}
\bigg(\prod_{l=k+1}^{k+j}
|\lambda_l|^2-\prod_{l=k}^{k+j-1} |\lambda_l|^2\bigg),
\quad k\Ge 0.
   \end{align*}
By \eqref{hdyw}, the terms of the above sum are
nonnegative, so
   \begin{align*}
\prod_{l=k}^{k+j-1} |\lambda_l|^2 =
\prod_{l=k+1}^{k+j} |\lambda_l|^2, \quad k\Ge 0, \,
j=1, \ldots, n-1.
   \end{align*}
Substituting $j=1$, we see that
$|\lambda_k|=|\lambda_{k+1}|$ for all $k\Ge 0$, so the
implication (i)$\Rightarrow$(ii) is proved. The
converse implication is straightforward.
   \end{proof}
Now we are in a position to give an example of an operator of
class $\mathscr{S}$ whose $n$th power is not of class
$\mathscr{S}$ for every integer $n\Ge 2$.
  \begin{exa} \label{rozwis}
Let $W\in \ogr{\ell^2}$ be a subnormal unilateral weighted
shift with nonzero weights $\{\lambda_n\}_{n=0}^{\infty}$.
Denote by $\mu$ the Berger measure of $W$, that is, $\mu$ is
a unique Borel probability measure on the closed half-line
$[0,\infty)$ such that
   \begin{align*}
\|W^k e_0\|^2 = \int_{[0,\infty)} x^k d\mu(x), \quad
k\Ge 0.
   \end{align*}
It is easy to see that
   \begin{align} \label{konswr}
   \begin{minipage}{68ex}
{\em the sequence $\{|\lambda_n|\}_{n=0}^{\infty}$ is
constant if and only if the Berger measure $\mu$ of
$W$ is supported in a singleton $\{\theta\}$ with
$\theta\in (0,\infty)$.}
   \end{minipage}
   \end{align}
It is also well known that the sequence
$\{|\lambda_n|\}_{n=0}^{\infty}$ is constant if and
only if $W$ is unitarily equivalent to the positive
multiple of the (isometric) unilateral shift on
$\ell^2$ (this is true for arbitrary injective
unilateral weighted shifts, see \cite{shi74}). Set
$\hh_1=\hh_2=\ell^2$ and take two isometries $V,E\in
\ogr{\ell^2}$ such that $\ob{V} \perp \ob{E}$. Then
clearly $T:=\big[\begin{smallmatrix} V & E \\
0 & W\end{smallmatrix}\big] \in \gsbh$. Since subnormal
operators are hyponormal, we deduce from \eqref{konswr} and
Propositions~\ref{xyzzyv} and \ref{suwt} that for every
integer $n\Ge 3$, the following conditions are equivalent:
   \begin{enumerate}
   \item[$\bullet$] $T^n\in \gsbh$,
   \item[$\bullet$] $T^2\in \gsbh$,
   \item[$\bullet$] the Berger measure of $W$ is
supported in a singleton $\{\theta\}$ with $\theta >
0$.
   \end{enumerate}
In particular, if the Berger measure of $W$ is not supported
in a singleton subset of $(0,\infty)$, then by
Remark~\ref{aabbaa}, $|E_n| \not\subseteq |\tilde E_n|$ for
all integers $n\Ge 2$.
   \hfill $\diamondsuit$
   \end{exa}

   \end{document}